\documentclass[pdflatex,final,10pt]{article}
\usepackage[dvips]{epsfig}
\usepackage{graphics}
\usepackage{latexsym}
\usepackage{verbatim}
\usepackage{amsmath}
\usepackage{amssymb}
\usepackage{comment}
\usepackage{units}
\usepackage{amsfonts}
\usepackage{bm}
\usepackage{graphicx}
\usepackage{subcaption}
\usepackage{placeins}
\usepackage{booktabs}
\usepackage{xcolor}
\usepackage{mathtools}
\usepackage{tabularx}
\usepackage{array}
\usepackage{booktabs}
\usepackage{soul}
\usepackage{comment}
\usepackage[]{hyperref}
\hypersetup{
    colorlinks=true,%
    filecolor=black,%
    linkcolor=black,%
    urlcolor=black
}
\usepackage{mathtools}
\graphicspath{{fig/}}	% Root directory of the pictures 
\usepackage{amsthm}

\newtheorem{proposition}{Proposition}[section]
\newtheorem{remark}{Remark}[section]
\newtheorem{lemma}{Lemma}[section]
\newtheorem{definition}{Definition}[section]
\newcommand{\Fabrice}[1]{#1}
\newcommand{\Liu}[1]{#1}
\newcommand{\Del}[1]{#1}

\newenvironment{del}
  {}
  {}

\newcommand{\IMT}{Universit\'e de Toulouse; UPS, INSA, UT1, UTM, Institut de Math\'ematiques de Toulouse,
CNRS, Institut de Math\'ematiques de Toulouse UMR 5219, F-31062 Toulouse, France; fabrice.deluzet@math.univ-toulouse.fr}
\newcommand{\HIT}{Corresponding author. School of Mathematics, Harbin Institute of Technology, 92 West Dazhi Street, Nan Gang District, Harbin, 150001, China;
yangchang@hit.edu.cn}
\newcommand{\HITlzy}{School of Mathematics, Harbin Institute of Technology, 92 West Dazhi Street, Nan Gang District, Harbin, 150001, China;
25B912034@stu.hit.edu.cn}
\newcommand{\Dollars}{This work was supported by National Natural Science Foundation of China 
grant 12371432.  Support from the "F\'ed\'eration de Fusion pour la 
Recherche par Confinement Magn\'etique" (FrFCM) in the frame of the 
project "NEMESIA: Numerical mEthods for Macroscopic models of 
magnEtized plaSmas and related anIsotropic equAtions" is also acknowledged.
  }

\begin{document}
%\headers{Multiscale methods for confined plasma fluid models}{C. Yang, F. Deluzet}
\title{An Asymptotic-Preserving Scheme for Plasma Simulation
    in Quasi-Neutral and Low-Mach-Number Regimes with Kinetic Upgrades
%\thanks{Submitted to the editors DATE.\funding{\Dollars}}}
\footnote{\Dollars}}
\author{Z. Liu\thanks{\HITlzy} \and F. Deluzet\thanks{\IMT}
\and C. Yang\thanks{\HIT}}
\newtheorem{assumption}{Assumption}
%\newsiamthm{prop}{Proposition}
%\newsiamremark{remark}{Remark}

\author{Zeyu Liu\thanks{\HITlzy} \and Fabrice Deluzet\thanks{\IMT} \and Chang Yang\thanks{\HIT}}

\hyphenation{bounda-ry rea-so-na-ble be-ha-vior pro-per-ties
cha-rac-te-ris-tic}

\maketitle

\begin{abstract}
 We propose an asymptotic-preserving micro--macro method bridging a kinetic description of electrons and a low-frequency reduced model in which the electrons are a massless, quasi-neutral fluid obeying the Boltzmann relation. Two features distinguish the construction. First, the fluid and low-Mach limits are coupled, so that the low-Mach stiffness is handled on a macroscopic system, where implicit treatment is affordable, rather than on the kinetic equations. Second, an auxiliary variable rescales the stiff force balance, turning the singular
low-Mach limit into a regular limit of the augmented system, which is shown to remain non-degenerate uniformly in the Debye length as well. This matters at the discrete level: with an iterative linear solver, the stiffness induced by the small Mach number amplifies the solver residual, so that a scheme designed to be asymptotic-preserving in its time discretization alone loses that property once the full solution chain is taken into account. The proposed scheme retains it, with no tightening of the solver tolerance as the Mach number vanishes, and admits a post-processing variant that decouples the auxiliary variable and reduces the size of the linear system. Numerical experiments spanning distinct parameter regimes confirm the analysis: standard semi-implicit schemes lose low-Mach-number equilibrium under residual amplification, whereas the proposed schemes preserve it down to round-off. 
\end{abstract}

\vspace{0.1cm}


\tableofcontents

\section{Introduction}
Plasma dynamics couples microscopic particle transport to macroscopic collective behavior across widely separated scales. In magnetic-confinement fusion, boundary plasmas, and space plasmas, the electron distribution often departs from a local Maxwellian: weak collisions, strong fields, boundary effects, and steep gradients generate velocity-space anisotropy and higher-order moment structures that fluid closures cannot resolve. Kinetic models such as the Vlasov and Vlasov--BGK equations, coupled to self-consistent fields, describe these effects directly, and are essential in scenarios such as pedestal transport and kinetic instabilities \cite{wan2012global}, divertor sheaths and scrape-off-layer transport \cite{stangeby2000plasma}, and non-Maxwellian electrons during magnetic reconnection \cite{munoz2016nonmaxwellian}. Their direct simulation is nonetheless costly and stiff: the distribution lives in high-dimensional phase space, and the Debye length, plasma period, collision time, ion-to-electron mass ratio and transport scales coexist in a single problem, straining both particle-in-cell \cite{birdsall1991particle,hockney1988computer} and deterministic phase-space methods \cite{yang2021highly}.

In many regions of interest, however, the electrons are in fact strongly collisional, quasi-neutral, and essentially inertia-less, and are then accurately described by a much cheaper reduced model: a massless, quasi-neutral and isothermal fluid closed by the Boltzmann relation between density and potential. The aim of this work is to build a numerical method that bridges these two descriptions: the full kinetic Vlasov--BGK--Poisson system and its low-frequency isothermal massless-electron, possibly quasi-neutral limit, and lets the simulation move from one to the other where each is appropriate. Following the asymptotic-transition analysis of \cite{crestetto2020bridging}, we design a scheme that mimics, at the discrete level, the singular limits performed on the continuous equations to pass from the kinetic to the reduced model. Because these limits are singular, we rely on asymptotic-preserving (AP) methods \cite{jin1999efficient,jin2010asymptotic,degond2017asymptotic}, whose stability is not constrained by the vanishing parameters: the collision parameter (fluid limit), the Mach number (low-Mach or inertia-less limit), and the Debye length
(quasi-neutral limit) can each be sent to zero to reach the cheaper model, without adapting the discretization for stability or consistency. This continues a line of AP work on the quasi-neutral and fluid limits of BGK--Vlasov--Poisson \cite{crouseilles2016multiscale,crispel2007asymptotic,degond2010asymptotic} and on all-speed, low-Mach formulations \cite{cordier2012asymptotic,dimarco2018second}, and it uses a micro--macro decomposition \cite{lemou2008new,bennoune2008uniformly} that splits the distribution function into a Maxwellian, carrying the conserved moments and the constant electron temperature, and a microscopic perturbation orthogonal to the collision invariants.

A distinctive feature of our construction is that the fluid limit is coupled to the low-Mach (inertialess-electron) limit through the scaling $\kappa = o(M^2)$. Rather than treating the low-Mach stiffness at the kinetic level, where an implicit solve would be prohibitively expensive in phase space, this coupling relaxes the distribution to its fluid limit and transfers the stiffness onto a macroscopic system, where implicit treatment is far cheaper. The reduction is thus performed where it is affordable.

The coupled limits raise two distinct degeneracies. In the quasi-neutral limit, Poisson equation collapses to an electroneutrality constraint and no longer determines the potential; in the low-Mach limit, the inertial term of the momentum equation vanishes and velocity evolution degenerates into a force balance. Beyond these continuous degeneracies, we analyze the AP property at the fully discrete level, including the linear solve. This step is usually left out, as the literature is largely confined to one-dimensional problems where a direct solver returns residuals near machine precision and the issue never appears. With an iterative solver, by contrast, the $1/M^2$ stiffness factor in the momentum equation amplifies the residual, so that the convergence tolerance must be tightened as $M \to 0$ merely to keep the velocity update accurate: otherwise the amplified residual accumulates in time and a scheme that is AP in its time discretization alone loses that property once the full solution chain
is accounted for.

Accordingly, the contributions of this work are as follows. Starting from a dimensionless Vlasov--BGK--Poisson system, we derive an equivalent micro--macro model in which the macroscopic momentum equation is coupled to the microscopic component through a kinetic pressure correction, its second velocity moment. We then analyze the asymptotic limit of this model in the combined fluid and low-Mach regime, and identify the conditions on the two parameters: the collision parameter $\kappa$ and the Mach number $M$, coupled through $\kappa = o(M^2)$, under which it reduces to the Boltzmann equilibrium in the joint limit. Analyzing the quasi-neutral and low-Mach degeneracies separately, we show that they arise from distinct mechanisms. %
For the low-Mach regime, we introduce an augmented micro--macro formulation built on an auxiliary variable. In the original system, the balance among pressure gradient, electric force, and kinetic pressure is multiplied by a factor that grows without bound as the Mach number vanishes, making the low-Mach limit singular. The auxiliary variable is defined to rescale this stiff term, so that the low-Mach limit becomes a regular limit of the augmented system, well defined down to $M = 0$. %
The resulting scheme is asymptotic-preserving in the low-Mach limit along the entire discrete solution chain, requiring no tightening of the linear-system tolerance as $M \to 0$. Moreover, the same augmented formulation remains non-degenerate uniformly in the quasi-neutral limit. We further recast the scheme as a post-processing of a semi-implicit discretization, decoupling the auxiliary variable from the other macroscopic unknowns and thereby reducing the size of the linear system. We validate the method on three test cases chosen to span distinct parameter regimes. Linear Landau damping probes the kinetic, non-quasi-neutral regime, away from the
low-Mach limit, and checks that the AP scheme recovers a standard kinetic description. The low-Mach equilibrium test targets the opposite extreme, the singular limit corresponding to the Boltzmann equilibrium, where it isolates the residual-amplification mechanism that makes IMEX diverge while the AP and AP-post schemes stay at round-off. Plasma expansion into vacuum finally combines these regimes in a single self-consistent problem, with local breakdown of quasi-neutrality, kinetic effects, and a low-Mach bulk coexisting with a non-low-Mach extracted beam, providing a complementary and more demanding validation against a Particle-In-Cell reference.

The paper is organized as follows. Section~2 presents the nondimensionalization,
the micro--macro decomposition, the limit models, and their degeneracy
structures. Section~3 develops the semi-implicit AP discretization and its
low-Mach consistency. Section~4 reports the numerical experiments, and
Section~5 concludes.
\section{Modelling}
\subsection{Vlasov--BGK--Poisson system}
We start from the Vlasov--BGK--Poisson system for electrons,
\begin{equation} \label{1}
    \left \{
    \begin{aligned}
        &\partial_t{f_e} + v\cdot \nabla_x{f_e}-\frac{e}{m_e}(E+v \times B) \cdot \nabla_v{f_e}=\Del{\nu_{ee}(\mathcal{M}_{e}-f_e)},\\
        &-\Delta_{x}\phi=\frac{e}{\varepsilon_{0}}(n_{i}-n_e),
    \end{aligned}
    \right.
\end{equation}
%where
%\begin{equation*}
%    \mathcal{Q}_{ee}=\nu_{ee}(\mathcal{M}_{e}-f_e).
%\end{equation*}
Here $f_e(x,v,t)$ is the electron distribution function, $E$ and $B$ denote the electric and magnetic fields, respectively, and $\phi$ is the electrostatic potential, with $E=-\nabla_x\phi$. Moreover, $e>0$ is the magnitude of the elementary charge, so that the electron charge is $-e$, $m_e$ is the electron mass, and $\varepsilon_0$ is the vacuum permittivity. The \Del{source term in the kinetic equation} is the BGK collision operator describing electron--electron collisions, and $\nu_{ee}$ is the electron--electron collision frequency, given by
\begin{equation*}
    \nu_{ee} = K_0 \frac{n_{e}}{(k_{B}T_{e})^{3/2}} \frac{\sqrt{2}}{\sqrt{m_{e}}}, \quad 
    K_0 = C \left( \frac{q^2}{4\pi\varepsilon_0} \right) \ln(\Lambda),
\end{equation*}
where $C$ is a constant of order one and $\ln(\Lambda)$ is the Coulomb logarithm.

The electron Maxwellian distribution function, \Fabrice{denoted as $\mathcal{M}_{e}$ and associated with $f_e$, is defined as}
\begin{equation*}
    \mathcal{M}_{e}(x,v,t) = n_e(x,t) \left( \frac{m_e}{2\pi k_B T_e(x,t)} \right)^{\frac{D_v}{2}} \exp \left( -\frac{m_e|u_e(x,t)-v|^{2}}{2k_BT_e(x,t)} \right),
\end{equation*}
where $D_v$ is the dimension of the velocity space and $k_B$ is the Boltzmann constant. The Maxwellian parameters $n_e$, $u_e$, and $T_e$ represent the electron density, mean velocity, and temperature, respectively, \Fabrice{and are defined as the moments of the distribution function $f_e$,}
\begin{equation*}
    n_e=\int_{\Omega_v}f_e\,d v,\quad
    n_e u_e=\int_{\Omega_v}v f_e\,dv,\quad
    \frac{1}{\gamma-1}n_e k_BT_e=\frac{m_e}{2}\int_{\Omega_v}|v-u_e|^2 f_e\,dv.
\end{equation*}
The ratio of specific heats depends on the velocity dimension through $\gamma-1 = {2/D_v}$.

\subsection{Dimensionless analysis}
\Fabrice{The purpose here is to introduce dimensionless parameters related to the different physical asymptotics: the low-Mach regime, the quasi-neutral assumption, and the anisotropy induced by the magnetic field intensity. To this end, we nondimensionalize the system.} Let $T_0$ and $n_0$ be characteristic scales of electron temperature and density, $\phi_0$ and $B_0$ be characteristic scales of electric potential and magnetic field, and $x_0$ and $t_0$ be characteristic length and time scales. \Fabrice{The speed of interest, defined as $\vartheta_0 = x_0 / t_0$, represents the observation scale chosen to analyze the system. Rather than an intrinsic property of the plasma, it reflects the specific physical phenomena under consideration. Depending on the chosen space and time scales, $\vartheta_0$ can be tuned to track macroscopic fluid transport, wave propagation, or thermal particle diffusion. The microscopic velocity is naturally scaled by the electron thermal velocity, which is the intrinsic physical scale for particle agitation. Finally, the characteristic electric field is related to the typical scale of the electrostatic potential, yielding}
\begin{equation*}
    v_0^2=\frac{k_BT_0}{m_e},\quad E_0=\frac{\phi_0}{x_0},\quad \vartheta_0=\frac{x_0}{t_0}.
\end{equation*}
We use the following nondimensional variables:
\begin{equation} \label{2}
    \left\{
    \begin{aligned}
        &x^{*}=\frac{x}{x_0},\quad t^{*}=\frac{t}{t_0},\quad v^{*}=\frac{v}{v_0},\quad f_e^{*}=\frac{f_e}{n_0/(v_0)^{D_v}},\quad n_e^{*}=\frac{n_e}{n_0},\quad n_{i}^{*}=\frac{n_{i}}{n_{0}},\\
        &T_e^*=\frac{T_e}{T_0},\quad {u_e^*=\frac{u_e}{u_0}}, \quad E^{*}=\frac{E}{E_0},\quad B^{*}=\frac{B}{B_0},\quad \phi^{*}=\frac{\phi}{\phi_{0}}.
    \end{aligned}
    \right.
\end{equation}

By substituting the characteristic scales into the dimensional equations and keeping the same notations for the non-dimensional variables to simplify notation, we obtain the general dimensionless electron Vlasov--BGK--Poisson system:
\begin{equation} \label{3}
    \left \{
    \begin{aligned}
        &\frac{\vartheta_0}{v_0}\partial_{t}f_e + v\cdot\nabla_{x}f_e - \eta E + \frac{\vartheta_0}{v_0} \Omega_e  \, v\times B \cdot \nabla_{v}f_e = \frac{1}{\kappa_0} \frac{\vartheta_0}{v_0} \nu_{ee} \left(\mathcal{M}_{e}-f_e \right), \\
        &-\lambda^{2}\eta\Delta_x\phi = n_{i}-n_e,
    \end{aligned}
    \right.
\end{equation}
  \Fabrice{At this stage, no assumptions have been made regarding the observation scales or the specific magnitudes of the physical quantities. The dimensionless parameters appearing in \eqref{3} are defined as follows: $\lambda = \lambda_D / x_0$ is the dimensionless Debye length (where $\lambda_D = \sqrt{\varepsilon_0 k_B T_0 / e^2 n_0}$); $\eta = e\phi_0 / k_B T_0$ represents the electric energy normalized by the electron internal energy; $\Omega_e = t_0 e B_0 / m_e$ characterizes the number of electron cyclotron periods over the typical time $t_0$; and $\kappa_0^{-1} = t_0 \nu_{ee,0}$ corresponds to the number of electron collisions occurring during this same typical time.}

\Fabrice{To study specific physical regimes, we must carefully select the observation scales and formulate appropriate assumptions regarding the dimensionless parameters. In this context, two characteristic speeds naturally emerge: the microscopic velocity $v_0$, associated with the thermal agitation of the particles, and the macroscopic velocity $u_0$, which governs the evolution of the plasma bulk. The electron Mach number $M$ is defined as the ratio of these two quantities, $M = u_0 / v_0$. The low-Mach regime is therefore characterized by the asymptotic limit $M \to 0$. The regime under investigation is therefore related to the assumption that the plasma evolution is mainly governed by the mean velocity $u_0$ while the thermal velocity of particles is large compared to this scale. The objective there is to follow the evolution of the system without resolving the fast particle velocity hence setting $\vartheta_0= u_0$. Inserting this assumption into the dimensionless system yields}
\begin{equation} \label{5}
    \left \{
    \begin{aligned}
        &M\partial_{t}f_e+v\cdot\nabla_{x}f_e-\eta E\cdot \nabla_{v}f_e-\Omega_{e}Mv\times B\cdot \nabla_{v}f_e=\frac{M}{\kappa}\left(\nu_{ee}(\Del{\mathcal{M}_{e}}-f_e) \right),\\
        &-\lambda^{2}\eta\Delta_{x}\phi=n_{i}-n_e.
    \end{aligned}
    \right.
\end{equation}
where the dimensionless Maxwellian is
\begin{equation*}
    \mathcal{M}_{e}(x,v,t)=n_e(x,t)\left(\frac{1}{2\pi T_e(x,t)}\right)^{\frac{D_v}{2}}
    \exp\left( -\frac{|Mu_e(x,t)-v|^{2}}{2T_e(x,t)}\right),
\end{equation*}
with the dimensionless parameters defined in Table~\ref{4}.
% \begin{table}[!h]
%     \centering
%     \caption{Dimensionless parameters.} \label{4}
%     \begin{tabular}{ll}
%     \toprule
%         Dimensionless parameter & Definition \\
%         \midrule
%         $\lambda=\lambda_{D}/{x_{0}}$ & dimensionless Debye length \\
%         $M=u_0/v_0$ & electron Mach number \\
%         $\eta=(e\phi_0)/(k_BT_0)$ & ratio of electric energy to electron internal energy \\
%         $\Omega_e=(eB_0t_0)/m_e$ & number of electron cyclotron periods over the characteristic time \\
%         $\kappa^{-1}=\nu_{ee,0}t_0$ & number of electron--electron collisions over the characteristic time \\
%         \bottomrule
%     \end{tabular}
% \end{table}
\begin{table}[!ht]
    \centering
    \caption{Dimensionless parameters.}
    \label{4}
    \begin{tabularx}{\linewidth}{
        @{}
        >{\raggedright\arraybackslash}p{0.31\linewidth}
        >{\raggedright\arraybackslash}X
        @{}
    }
        \toprule
        Dimensionless parameter & Definition \\
        \midrule
        $\lambda=\lambda_D/x_0$
        & Dimensionless Debye length. \\

        $M=u_0/v_0$
        & Electron Mach number. \\

        $\eta=e\phi_0/(k_BT_0)$
        & Ratio of the electric energy to the electron internal energy. \\

        $\Omega_e=eB_0t_0/m_e$
        & Number of electron cyclotron periods over the characteristic time. \\

        $\kappa^{-1}=\nu_{ee,0}t_0$
        & Number of electron--electron collisions over the characteristic time. \\
        \bottomrule
    \end{tabularx}
\end{table}

\Fabrice{The low-Mach limit introduces severe structural stiffness into equation \eqref{5}.
Addressing this low-Mach-number stiffness directly at the kinetic level generally calls for an implicit discretization of the kinetic equation. Because of the high dimensionality of phase space, such implicit kinetic solvers are computationally demanding, and their cost becomes prohibitive in multidimensional settings. The strategy adopted here is instead to treat the low-Mach-number limit at the fluid level rather than at the kinetic level, by coupling it to the collisional (fluid) limit of the model. The gain is substantial, since an implicit discretization of a macroscopic (fluid) system is far cheaper than the corresponding implicit kinetic solver. The price to pay is a loss of accuracy in the low-Mach regime, where the kinetic corrections carried by the non-equilibrium perturbation are discarded as the system relaxes to its fluid limit.}

\Fabrice{Concretely, the low-Mach-number limit is governed by the Mach number $M$, while the relaxation toward the fluid limit is governed by the collision parameter $\kappa$. Coupling the two limits through the scaling $\kappa \sim M^2$ makes $\kappa$ vanish faster than $M$, so that the collisional relaxation rate $\sim \nu_{ee}/\kappa$ grows and the system is driven to its fluid limit at least as fast as the low-Mach degeneracy develops. In this strongly collisional, low-Mach regime, $\kappa \sim M^2 \to 0$, the dominant collision operator formally enforces $f_e \to \mathcal{M}_{e}$ and reduces the system to its fluid limit. This ordering of the two limits is what secures the stability of the numerical method in the joint limit and lets us avoid analyzing the low-Mach degeneracy directly at the kinetic level. This feature is investigated in the next section, introducing a Micro-Macro decomposition of the distribution function in order to easily recover the macroscopic model within the kinetic one.}

\Fabrice{This coupling is introduced primarily for numerical stability, and tying the low-Mach and fluid limits together is not, in general, physically justified. It nonetheless admits a reasonable modeling interpretation. The electron collision frequency increases with density and decreases with temperature, $\nu_{ee} \propto n_e T_e^{-3/2}$, whereas the low-Mach regime is naturally associated with hot plasmas, whose large sound speed lowers the Mach number. The scaling $\kappa \sim M^2$ is thus consistent with a physical setting in which the low-Mach limit is attained within a sufficiently dense plasma, where the high density sustains strong collisionality despite the high temperature. In this picture, the relevant transition connects a dense, hot, collisional region to a less dense and cooler one, as encountered, for instance, near plasma sheaths.}

\Fabrice{Beyond modeling aspects, the primary objective of the proposed methodology lies in model coupling: bridging a fluid representation in the low-Mach limit\Del{,} commonly known as the Boltzmann approximation\Del{,} with a kinetic description when the fluid reduction fails. This paper is therefore devoted to constructing such a consistent transition model.}
\subsection{Derivation of the micro--macro model}
We now introduce \Fabrice{a decomposition of the electron distribution function into a global Maxwellian, characterized by a constant macroscopic temperature $T_{e,0}$, and a kinetic deviation $g_e$, yielding:}
\begin{equation} \label{6}
    f_e=\mathcal{M}_{e}+g_{e},
\end{equation}
where
\begin{align}\label{eq:Maxwellian:adim}
     \mathcal{M}_{e}\Fabrice{(x,v,t)}=n_e(x,t)\left(\frac{1}{2\pi T_{e,0}}\right)^{\frac{D_v}{2}}
    \exp\left( -\frac{|Mu_e(x,t)-v|^{2}}{2T_{e,0}}\right).
\end{align}
By construction, $g_e$ satisfies the moment constraints
\begin{equation*}
    \left\langle
    \begin{pmatrix}
        1\\
        v
    \end{pmatrix} g_{e}
    \right\rangle
    =\int_{\Omega_v}\begin{pmatrix}1\\v\end{pmatrix} g_{e}\,dv=0.
\end{equation*}
Substituting the \Fabrice{ansatz \eqref{6} into the Vlasov equation} \eqref{5} and taking zeroth and first moments over $\Omega_v$ yields the Euler-type macroscopic equations
\begin{equation}
    \partial_t n_e+\nabla_x\cdot(n_eu_e)=0,
    \label{7}
\end{equation}

\begin{equation}
\begin{aligned}
    M^2\left(
        \partial_t(n_eu_e)
        +\nabla_x\cdot(n_eu_e\otimes u_e)
    \right)
    &+T_{e,0}\nabla_x n_e
    -\eta n_e\nabla_x\phi
    \\
    \qquad
    &+\Omega_eM^2n_eu_e\times B
    =
    -\nabla_x\cdot
    \left\langle v\otimes v\,g_e\right\rangle.
\end{aligned}
\label{8}
\end{equation}
\Fabrice{The source term in the momentum equation \eqref{8} accounts for kinetic corrections to the fluid model, including contributions stemming from the fluctuations around the reference temperature $T_{e,0}$.}

Following \cite{bennoune2008uniformly} an orthogonal basis of the \Fabrice{null space of the collision operator} is
\begin{equation} \label{9}
    \mathcal B=\left\{ \frac{1}{n_e}\mathcal{M}_{e}, \frac{(v-Mu_e)}{\sqrt{T_{e,0}}}\frac{1}{n_e}\mathcal{M}_{e} \right\}.
\end{equation}
The corresponding projection operator onto this space is
\begin{equation} \label{10}
    \Pi_{\mathcal{M}}(\varphi)=\frac{1}{n_e}\left[ \langle\varphi\rangle+\frac{(v-Mu_e)}{T_{e,0}}\cdot\langle(v-Mu_e)\varphi\rangle \right]\mathcal{M}_{e}.
\end{equation}

The following elementary properties will be used \Fabrice{for the derivation of the micro-macro set of equations.}
\begin{lemma}
Let $\mathcal{M}_e$ and $g_e$ be defined by \eqref{6}. Then
\begin{align*}
        &\Pi_{\mathcal{M}}(\mathcal{M}_{e})=\mathcal{M}_{e}, \quad \Pi_{\mathcal{M}}(\partial_t\mathcal{M}_{e})=\partial_t\mathcal{M}_{e},\quad \Pi_{\mathcal{M}}(\nabla_v\mathcal{M}_{e})=\nabla_v\mathcal{M}_{e},\\
        &\Pi_{\mathcal{M}}(g_e)=0, \quad \Pi_{\mathcal{M}}(\partial_tg_e)=0,\quad \Pi_{\mathcal{M}}(\nabla_v g_e)=0.
\end{align*}
\end{lemma}

\begin{proof}
By the definition of the dimensionless Maxwellian,
\begin{align*}
        \Pi_{\mathcal{M}}(\mathcal{M}_{e}) &= \frac{1}{n_e}\left[ \langle\mathcal{M}_{e}\rangle+\frac{(v-Mu_e)}{T_{e,0}}\cdot\langle(v-Mu_e)\mathcal{M}_{e}\rangle \right]\mathcal{M}_{e}\\
        &=\frac{1}{n_e}\left[ n_e +\frac{(v-Mu_e)}{T_{e,0}}\cdot(Mn_eu_e-Mn_eu_e) \right]\mathcal{M}_{e}=\mathcal{M}_e.
\end{align*}
Similarly,
\Del{\begin{align*}
        \Pi_{\mathcal{M}}(\partial_t\mathcal{M}_{e}) &= \frac{1}{n_e}\left[ \partial_tn_e+\frac{(v-Mu_e)}{T_{e,0}}\cdot(M\partial_t(n_eu_e) - Mu_e\partial_tn_e) \right]\mathcal{M}_{e} %\\
        %&=\frac{\partial_tn_e}{n_e}\mathcal{M}_e + M\frac{(v-Mu_e)}{T_{e,0}}(\partial_tu_e)\mathcal{M}_e
        =\partial_t\mathcal{M}_e.
\end{align*}}
The identity $\Pi_{\mathcal{M}}(\nabla_v\mathcal{M}_e)=\nabla_v\mathcal{M}_e$ follows in the same way, using integration by parts. Finally, the properties involving $g_e$ follow immediately from $\langle g_e\rangle=0$ and $\langle vg_e\rangle=0$.
\end{proof}

Applying $\mathbb{I}-\Pi_{\mathcal{M}}$ to the kinetic equation in \eqref{5} gives the microscopic equation \Del{driving the evolution of $g_e$,}
%\begin{equation} \label{11}
%    \begin{split}
%        M\partial_t{g_e}&+(\mathbb{I}-\Pi_{\mathcal{M}})(v\cdot\nabla_xg_e-\Omega_eMv\times B\cdot\nabla_vg_e)-\eta E\cdot\nabla_vg_e \\
%        &\hspace{1cm}=-\frac{M}{\kappa}\nu_{ee}g_e-(\mathbb{I}-\Pi_{\mathcal{M}})(v\cdot\nabla_x\mathcal{M}_{e}-\Omega_eMv\times B\cdot\nabla_v\mathcal{M}_{e}).
%    \end{split}
%\end{equation}
together with \eqref{7}, \eqref{8}, and Poisson equation, this forms a closed \Fabrice{system, referred to as }micro--macro system\Del{, specified in the following proposition.}

\begin{proposition}[\Del{Equivalence of the micro–macro and Vlasov–BGK–Poisson systems}]\label{prop:equivalence}
The micro--macro system defined as 
\begin{align}
    &\partial_{t}{n_e}+\nabla_{x}\cdot(n_e u_e)=0,\label{MMd}\\
    \begin{split}
    &M^{2}(\partial_{t}(n_e u_e)+\nabla_{x}\cdot(n_e u_e\otimes u_e))+T_{e,0}\nabla_{x}n_e-
    \eta n_e\nabla_{x}\phi\\
    &\hspace{5.5cm}+\Omega_{e}M^{2}n_e u_e\times B = -\nabla_{x}\cdot\langle v\otimes vg_{e}\rangle,
    \end{split}\label{MMm}\\
    \begin{split}
    &M\partial_t{g_e}+(\mathbb{I}-\Pi_{\mathcal{M}})(v\cdot\nabla_xg_e-\Omega_eMv\times B\cdot\nabla_vg_e)-\eta E\cdot\nabla_vg_e \\
    &\hspace{2cm}=-\frac{M}{\kappa}\nu_{ee}g_e-(\mathbb{I}-\Pi_{\mathcal{M}})(v\cdot\nabla_x\mathcal{M}_{e}-\Omega_eMv\times B\cdot\nabla_v\mathcal{M}_{e}),
    \end{split}\label{MMk}\\
    &-\lambda^{2}\eta\Delta_{x}\phi=n_{i}-n_e,\label{MMp}
\end{align}
is equivalent to the dimensionless Vlasov--BGK--Poisson system \eqref{5}.
\end{proposition}
\begin{proof}
The Poisson equations in the two formulations are identical, so it remains to prove the equivalence of the kinetic equation and the micro--macro \Fabrice{set of equations, namely equations \eqref{7}, \eqref{8} for the macroscopic quantities and equation \eqref{MMk} for the kinetic complement}. Applying the projection $\Pi_{\mathcal{M}}$ to the kinetic equation in \eqref{5} yields the macroscopic projected equation
\begin{equation*}
        \begin{split}
            M\partial_t\mathcal{M}_e&+\frac{1}{n_e}v(\partial_xn_e)\mathcal{M}_e+M(\partial_xu_e)\mathcal{M}_e-\frac{1}{n_e}\partial_x\langle v^2g_e\rangle\partial_v\mathcal{M}_e-M^2u_e(\partial_xu_e)\partial_v\mathcal{M}_e\\
            &=\eta E\,\partial_v\mathcal{M}_e+\Omega_eM^2(u_e\times B)\partial_v\mathcal{M}_e.
        \end{split}
\end{equation*}
\Fabrice{The parameters of the Maxwellian satisfy the zeroth and first moments of this kinetic equation, yielding equations} \eqref{7} and \eqref{8}. Since equation \eqref{MMk} is obtained by applying $\mathbb{I}-\Pi_{\mathcal{M}}$ to the same kinetic equation, the projected macroscopic equations and the microscopic equation together are equivalent to the original dimensionless Vlasov--BGK equation. This \Fabrice{concludes the proof}.
\end{proof}

\subsection{Asymptotic limits}
The purpose here is to emphasize the singular nature of both the quasi-neutral and the low-Mach limits. The degeneracy of the system is first highlighted before a strategy to circumvent this difficulty is proposed.
\subsubsection{Quasi-neutral limit}
\Fabrice{The quasi-neutral limit is reached when the characteristic scale of charge separation, namely the Debye length, is much smaller than the system's characteristic length, which corresponds to the asymptotic limit $\lambda\to0$. In this regime, the governing equations remain unchanged, except for Poisson equation, which degenerates into the local charge neutrality condition $n_e=n_i$.}
%\begin{align*}
%    &\partial_{t}{n_e}+\nabla_{x}\cdot(n_e u_e)=0,\\
%    &M^{2}(\partial_{t}(n_e u_e)+\nabla_{x}\cdot(n_e u_e\otimes u_e))+T_{e,0}\nabla_{x}n_e-
%    \eta n_e\nabla_{x}\phi+\Omega_{e}M^{2}n_e u_e\times B = -\nabla_{x}\cdot\langle v\otimes vg_{e}\rangle,\\
%    &M\partial_t{g_e}+(\mathbb{I}-\Pi_{\mathcal{M}})(v\cdot\nabla_xg_e-\Omega_eMv\times B\cdot\nabla_vg_e)-\eta E\cdot\nabla_vg_e \\
%    &\hspace{6.1cm}=-\frac{M}{\kappa}\nu_{ee}g_e-(\mathbb{I}-\Pi_{\mathcal{M}})(v\cdot\nabla_x\mathcal{M}_{e}-\Omega_eMv\times B\cdot\nabla_v\mathcal{M}_{e}),\\
%    &n_e=n_i.
%\end{align*}
%The Poisson equation degenerates into 
\Fabrice{Although consistent with the assumption of zero net charge, this algebraic relation no longer allows for a direct calculation of the electric potential $\phi$. To circumvent this, the continuity equation $\partial_t \rho + \nabla \cdot J = 0$ (derived from the species continuity equations) is routinely exploited to provide an alternative equation for $\phi$. This is achieved by imposing the current conservation condition, $\nabla \cdot J = 0$, which acts as an implicit constraint to determine the electric potential. This highlights the singular nature of the quasi-neutral limit: Poisson equation is substituted by a constraint originating from the kinetic equation itself, fundamentally tied to the plasma dynamics rather than to the electrostatic field. In this regard, the micro--macro approach is particularly well-suited, since the macroscopic moments of the Vlasov equation are inherently part of the model, thereby naturally providing the mathematical framework required to compute the electric field in the quasi-neutral limit.}

The following proposition provides the elliptic governing equation for the electric potential in the limit $\lambda\to0$, which replaces the degenerate Poisson equation.

\begin{proposition}[Reformulated Poisson equation and quasi-neutral model] \label{prop2.2}
Within the micro--macro framework, assuming that the initial conditions satisfy both Poisson equation and its first-order time derivative, namely:
\begin{equation} \label{eq:init_cond}
    -\lambda^2\eta\Delta_x\phi\vert_{t=0} = n_i - n_e\vert_{t=0} \quad \text{and} \quad -\lambda^2\eta\Delta_x\partial_t\phi\vert_{t=0} = \nabla_x\cdot(n_e u_e)\vert_{t=0},
\end{equation}
then Poisson equation is equivalent for all $t > 0$ to the following reformulated equation:
\begin{align} \label{eq:poisson_reformulated}
    -\nabla_x\cdot\Big((\Fabrice{\eta}n_e+M^2\lambda^2\eta\partial_t^2)\nabla_x\phi\Big) = &- M^2\nabla_x^2:(n_eu_e\otimes u_e)-T_{e,0}\Fabrice{\Delta_x}n_e \\ &\Fabrice{ - \Omega_e M^2 \nabla_x \cdot (n_e u_e \times B )} - \nabla_x^2:\left< v\otimes vg_e \right>. \notag
\end{align}
where $\nabla_x^2 : T = \nabla_x \cdot (\nabla_x \cdot T)$ denotes the double divergence of a second-order tensor $T$.

In the limit $\lambda\to0$, the reformulated Poisson equation yields the following elliptic equation for the electric potential, subject to suitable boundary conditions:
\begin{align*}
    -\nabla_x\cdot(\Fabrice{\eta}n_e\nabla_x\phi) = &- M^2\nabla_x^2:(n_eu_e\otimes u_e)-T_{e,0}\Fabrice{\Delta_x}n_e \\
   &\quad \Fabrice{ - \Omega_e M^2 \nabla_x \cdot (n_e u_e \times B )} - \nabla_x^2:\left< v\otimes vg_e \right>.
\end{align*}
\end{proposition}

\begin{proof} \label{prop2.2pf}
The proof consists in time differentiating the continuity equation $\partial_t \rho + \nabla_x \cdot J=0$ to establish a wave-like equation for the charge density, which is then combined with the second-order time derivative of Poisson equation. Under the assumption of stationary ions, the continuity equation reduces to the electron density conservation law.

Differentiating equation \eqref{MMd} with respect to time and taking the divergence of the momentum equation \eqref{MMm} yields:
\begin{align*}
        &\partial_t^2n_e + \partial_t\nabla_x\cdot(n_eu_e) = 0, \\
        &M^2\partial_t\nabla_x\cdot(n_eu_e) + M^2\nabla_x^2:(n_eu_e\otimes u_e)+T_{e,0}\Fabrice{\Delta_x}n_e  = \\
        & \hspace{3.5cm} \Fabrice{\eta} \nabla_x\cdot(n_e\nabla_x\phi)\Fabrice{ - \Omega_e M^2 \nabla_x \cdot (n_e u_e \times B )}-\nabla_x^2:\left< v\otimes vg_e \right>.
\end{align*}
Eliminating the cross-derivative term $\partial_t\nabla_x\cdot(n_eu_e)$ between these two relations yields:
\begin{equation} \label{eq:wave_ne}
    \begin{split}
        M^2\partial_t^2n_e + \Fabrice{\eta}\nabla_x\cdot(n_e\nabla_x\phi) &= M^2\nabla_x^2:(n_eu_e\otimes u_e)+T_{e,0}\Fabrice{\Delta_x}n_e \\
        &\quad\Fabrice{ + \Omega_e M^2 \nabla_x \cdot (n_e u_e \times B )} + \nabla_x^2:\left< v\otimes vg_e \right>.
    \end{split}
\end{equation}
The reformulated Poisson equation is obtained by substituting $M^2\partial_t^2n_e$ into the double time derivative of Poisson equation: $-M^2 \lambda^2 \partial_t^2 \Delta_x \phi = -M^2 \partial_t^2 n_e$ providing equation \eqref{eq:poisson_reformulated}.

Now, let us define the deviation from Poisson equation as $\mathcal{E}(t,x) = \lambda^2 \eta \Delta_x \phi + n_i-n_e$, which vanishes if Poisson equation holds true. Since the ions are at rest ($\partial_t^2 n_i = 0$), the second-order time derivative of $\mathcal{E}$ reads $M^2 \partial_t^2 \mathcal{E} = -M^2 \partial_t^2 n_e + M^2 \lambda^2 \eta \partial_t^2 \Delta_x \phi$. Substituting $M^2 \partial_t^2 n_e$ from \eqref{eq:wave_ne} into this relation yields:
\begin{equation}\label{eq:Poisson:deviation}
    \begin{multlined}[0.9\textwidth]
    M^2\partial_t^2 \mathcal{E} = \nabla_x\cdot\Big((\Fabrice{\eta}n_e+M^2\lambda^2\eta\partial_t^2)\nabla_x\phi\Big) \\ - M^2\nabla_x^2:(n_eu_e\otimes u_e)-T_{e,0}\Fabrice{\Delta_x}n_e \Fabrice{ - \Omega_e M^2 \nabla_x \cdot (n_e u_e \times B )} - \nabla_x^2:\left< v\otimes vg_e \right>.
    \end{multlined}
\end{equation}
Enforcing the reformulated Poisson equation \eqref{eq:poisson_reformulated} is mathematically equivalent to canceling the source term in \eqref{eq:Poisson:deviation}. Since this equation dictates the dynamics of the deviation, we obtain the homogeneous relation $M^2 \partial_t^2 \mathcal{E}(t,x) = 0$ for all $t > 0$. Consequently, the Poisson equation is satisfied for all times if and only if the deviation and its first derivative vanish at $t = 0$:
\begin{equation*}
    \mathcal{E}\vert_{t=0} = 0 \quad \text{and} \quad \partial_t \mathcal{E}\vert_{t=0} = 0,
\end{equation*}
which corresponds exactly to the initial conditions stated by equation \eqref{eq:init_cond}. 

Finally, it should be noted that the limit $\lambda \to 0$ inherently presupposes a strictly positive electron density ($n_e > 0$), which in turn guarantees the ellipticity of the reformulated equation in the quasi-neutral limit.
\end{proof}
\begin{remark}\label{remark:QN}
\Del{This continuous analysis is carried out to guide the time discretization, and calls for two comments. First, an AP scheme needs not be a direct discretization of the reformulated equations through second-order time derivatives. These derivatives are only an artifact of the continuous manipulation used to eliminate the cross terms between Poisson equation and the continuity equation. At the semi-discrete level, the implicit coupling of the continuity and momentum updates performs the same elimination directly, so that neither the wave-type reformulation nor time derivatives of Poisson equation need to be discretized as such.
Second, and more importantly, the reformulation shows that the electric field enters through two distinct contributions, one from Poisson equation and one from the electric-force term of the momentum equation. It is this double occurrence that renders the quasi-neutral limit regular, and the AP property therefore requires a semi-implicit discretization in which both occurrences of the electric field are treated implicitly.}
\end{remark}
\subsubsection{Asymptotic analysis of the joint low-Mach and fluid limits}

We first recall that the Boltzmann equilibrium is recovered from the micro--macro system through a joint fluid and low-Mach asymptotic regime. In strongly magnetized plasmas, this equilibrium holds along the magnetic field lines. Before restricting our attention to a simplified one-dimensional geometry along a magnetic field line, we first examine the behavior of the full multidimensional micro--macro model in this asymptotic regime to outline the difficulties raised by this asymptotic.

Formally taking the limit $M \to 0$ in equation \eqref{MMk}, the time-derivative term vanishes. This behavior reflects the physical nature of the low-Mach scaling: the characteristic macroscopic time scale is governed by the mean plasma bulk velocity, whereas the dynamics of the distribution function is genuinely driven by the microscopic (thermal) particle speed. Under the assumption that thermal velocities significantly exceed the mean velocity, the kinetic equation reduces to:
\begin{equation*}
\begin{split}
    &(\mathbb{I}-\Pi_{\mathcal{M}})\big(v\cdot\nabla_x g_e - \Omega_e M (v\times B)\cdot\nabla_v g_e\big) - \eta E\cdot\nabla_v g_e \\
    &\hspace{2cm} = -\frac{M}{\kappa}\nu_{ee}g_e - (\mathbb{I}-\Pi_{\mathcal{M}})\big(v\cdot\nabla_x\mathcal{M}_e - \Omega_e M (v\times B)\cdot\nabla_v\mathcal{M}_e\big).
\end{split}
\end{equation*}

At the macroscopic level, the full system simplifies to:
\begin{align*}
    &\partial_{t}{n_e} + \nabla_{x}\cdot(n_e u_e) = 0, \\
    &T_{e,0}\nabla_{x}n_e - \eta n_e\nabla_{x}\phi + \Omega_{e}M^{2}n_e u_e\times B = -\nabla_{x}\cdot\langle v\otimes v g_{e}\rangle,
\end{align*}
with Poisson equation remaining unchanged.

The macroscopic momentum equation is singular because it no longer allows for a direct determination of the parallel component of the momentum. To isolate this parallel direction, we align the magnetic field with the $x$-direction without loss of generality, thereby focusing the analysis on the 1D-1V parallel dynamics. The momentum balance yields:
\begin{equation}
    T_{e,0}\partial_{x}n_e - \eta n_e\partial_{x}\phi = -\partial_x \langle v^2 g_e \rangle.
\end{equation}
The parallel momentum no longer appears explicitly in this relation, which underscores the singular nature of the limit. To retrieve the Boltzmann equilibrium and restore a well-posed momentum equation, a fluid limit must also be enforced to eliminate the kinetic pressure $\langle v^2 g_e \rangle$. This is achieved by coupling the low-Mach and fluid limits through the scaling $\kappa = o(M^2)$ (e.g., $\kappa < M^2$). This coupling serves a dual purpose: first, it links the fluid asymptotic regime to the low-Mach limit to establish the desired force balance; second, it forces $g_e \to 0$ rapidly enough. Under this scaling, the kinetic equation for $g_e$ remains non-degenerate and can be recast as:
\begin{equation*}
\begin{split}
    \nu_{ee}g_e &= -M (\mathbb{I}-\Pi_{\mathcal{M}})\big(v\cdot\nabla_x g_e - \Omega_e M (v\times B)\cdot\nabla_v g_e\big) + M \eta E\cdot\nabla_v g_e \\
    &\quad - M (\mathbb{I}-\Pi_{\mathcal{M}})\big(v\cdot\nabla_x\mathcal{M}_e - \Omega_e M (v\times B)\cdot\nabla_v\mathcal{M}_e\big),
\end{split}
\end{equation*}
which guarantees that $g_e \to 0$ as $M \to 0$.

We now state the asymptotic properties of the micro--macro system in this combined regime. To simplify the subsequent analysis, the reformulated system is presented for a 1D-1V problem.

\begin{proposition}[Formal fluid low-Mach limit via Hilbert expansion] \label{prop2.3}
In the joint regime $(\kappa, M) \to (0, 0)$, under the scaling hypothesis $\kappa = o(M^2)$ (i.e., $\kappa / M^2 \to 0$), consider the Hilbert expansions in powers of $M^2$:
\begin{equation} \label{eq:hilbert_expansions}
    n_e = \sum_{k=0}^\infty (M^2)^k n_{e,k}, \quad 
    u_e = \sum_{k=0}^\infty (M^2)^k u_{e,k}, \quad 
    \phi = \sum_{k=0}^\infty (M^2)^k \phi_{k}.
\end{equation}
Injecting these expansions into the one-dimensional micro--macro model along the magnetic field lines yields, up to order $M^2$, the following closed formal limit system:
\begin{align}
        &\partial_t n_{e,0} + \partial_x (n_{e,0} u_{e,0}) = 0, \label{eq:M0:1}\\
        &T_{e,0}\partial_x n_{e,0} - \eta n_{e,0}\partial_x\phi_0 = 0, \label{eq:M0:2}\\
        &\partial_{t}(n_{e,0} u_{e,0})+\partial_{x}(n_{e,0} u_{e,0}^2) = \eta\big( n_{e,0}\partial_x\phi_{1}+n_{e,1}\partial_x\phi_{0}\big)-T_{e,0}\partial_x n_{e,1}, \label{eq:momentum:Hilbert}\\
        &g_e = 0, \label{eq:M0:4}\\
        &-\lambda^2 \eta \partial^2_{xx} \phi_{0} = n_i - n_{e,0}, \label{eq:M0:5a}\\
        &-\lambda^2 \eta \partial^2_{xx} \phi_{1} = - n_{e,1}. \label{eq:M0:5b}
\end{align}
\end{proposition}

The structure of the formal limit system \eqref{eq:M0:1}--\eqref{eq:M0:5b} calls for several remarks. First, the leading-order momentum equation \eqref{eq:M0:2} acts as a spatial constraint on $(n_{e,0}, \phi_0)$, recovering the classic Boltzmann relation and confirming the validity of the chosen scaling regime.

Second, the first-order Hilbert corrections $n_{e,1}$ and $\phi_1$ lack independent evolutionary dynamics (no continuity equation governs $\partial_t n_{e,1}$). Linked via Poisson relation $n_{e,1} = \lambda^2 \eta \partial^2_{xx} \phi_1$, they act as Lagrange-like multipliers in equation~\eqref{eq:momentum:Hilbert} to balance the advective momentum flux $\partial_t (n_{e,0} u_{e,0}) + \partial_x (n_{e,0} u_{e,0}^2)$.

Finally, solving the limit system directly through $(n_{e,1}, \phi_1)$ is not necessarily well suited for numerical schemes. This provides the primary motivation for introducing the augmented variable $L$ in Proposition~\ref{prop2.4}. Beyond reformulating this singular force imbalance into a well-posed auxiliary field without requiring explicit higher-order Hilbert corrections, introducing $L$ allows a single, unified set of augmented equations to seamlessly cover both the physical regime ($M \sim \mathcal{O}(1)$) and the asymptotic regime ($M \to 0$).

We now prove that the combined low-Mach and fluid limit leads to the limit model stated by Proposition~\ref{prop2.3}.

\begin{lemma}[Vanishing kinetic pressure contribution under $\kappa/M^2 \to 0$]
\label{lem:ge_estimate_1d}
Consider the 1D-1V setting along the magnetic field lines, with a bounded
velocity domain $\Omega_v = [-V_{\max}, V_{\max}]$, and assume that
\begin{enumerate}
    \item[(i)] the macroscopic fields $n_e$, $u_e$, $\phi$ are smooth, with derivatives bounded uniformly in $M$, and $\exists \,n_\ast>0,\  s.t.\ n_e\geq n_\ast$;
    \item[(ii)] the collision frequency satisfies $\exists\, \nu_\ast>0,\ s.t.\ \nu_{ee} \ge \nu_\ast$,
    with $\partial_x \nu_{ee}$ bounded uniformly in $M$;
    \item[(iii)] $g_e$ is smooth, and $g_e$, $\partial_t g_e$ and all their
    space and velocity derivatives are bounded in $L^\infty_v(\Omega_v)$
    uniformly with respect to $M$;
    \item[(iv)] the initial data are well prepared, $g_e\vert_{t=0} =
    \mathcal{O}(\kappa)$, so that the initial collisional layer of thickness
    $t \sim \kappa/\nu_{ee}$ is excluded.
\end{enumerate}
Then, under the coupled scaling $\kappa = o(M^2)$, the kinetic correction
satisfies $g_e = \mathcal{O}(\kappa)$, and the kinetic pressure gradient
\begin{equation} \label{eq:ge_bound_condensed}
    \partial_x \langle v^2 g_e \rangle = o(M^2)
\end{equation}
makes no contribution to the first-order parallel momentum balance
\eqref{eq:momentum:Hilbert}.
\end{lemma}
The detailed proof of Lemma~\ref{lem:ge_estimate_1d} is provided in Appendix~\ref{app:proof}.

\begin{proof}[Proof of Proposition~\ref{prop2.3}] \label{prop2.3pf}
We establish the formal limit system by inserting the expansions \eqref{eq:hilbert_expansions} into the governing equations.
Under the scaling $\kappa = o(M^2)$, the kinetic correction satisfies $g_e = o(M^2)$ by Lemma~\ref{lem:ge_estimate_1d}. Consequently, $g_{e,0} = 0$ and $g_{e,1} = 0$, so that the microscopic pressure contribution $\langle v^2 g_e \rangle = o(M^2)$ vanishes at both leading order and first order in the momentum balance.

The 1D parallel momentum equation reads:
\begin{equation*}
    M^2 \Big( \partial_t(n_e u_e) + \partial_x(n_e u_e^2) \Big) + T_{e,0} \partial_x n_e - \eta n_e \partial_x \phi + \partial_x \langle v^2 g_e \rangle = 0.
\end{equation*}
At leading order $\mathcal{O}(1)$, the advective and temporal terms (scaled by $M^2$) vanish along with the kinetic pressure gradient, yielding the Boltzmann relation \eqref{eq:M0:2}. At order $\mathcal{O}(M^2)$, gathering the advective and temporal terms together with the first-order fluid corrections $n_{e,1}$ and $\phi_1$ gives equation~\eqref{eq:momentum:Hilbert}.

Finally, expanding $\phi$ and $n_e$ in the Poisson relation $-\lambda^2 \eta \partial_{xx} \phi = n_i - n_e$ yields equations \eqref{eq:M0:5a} and \eqref{eq:M0:5b} at orders $\mathcal{O}(1)$ and $\mathcal{O}(M^2)$, respectively.
\end{proof}

As revealed by the formal asymptotic analysis in Proposition~\ref{prop2.3}, the fundamental difficulty of the low-Mach limit lies in a structural degeneracy: as $M \to 0$, the parallel momentum density $n_e u_e$ disappears from the leading-order balance, and its time evolution becomes governed by higher-order Hilbert corrections ($n_{e,1}, \phi_1$). 
To overcome this issue, the main objective is to design a unified formulation that naturally tracks the parallel momentum across all regimes, from standard kinetic flows ($M \sim \mathcal{O}(1)$) down to the asymptotic low-Mach limit ($M \to 0$), without suffering from singular degeneracy. 

The core idea consists in absorbing the stiff force imbalance (which formally vanishes as $M \to 0$) into a well-behaved auxiliary macroscopic field $L$, defined through its spatial gradient:
\begin{align}
    &\partial_{t}(n_e u_e)_x + \big(\nabla_{x}\cdot(n_e u_e\otimes u_e)\big)_x + \partial_x L = 0, \label{eq:reformulation_mom_1}\\
    &M^2 \partial_x L = T_{e,0}\partial_{x}n_e - \eta n_e\partial_{x}\phi + \big(\nabla_{x}\cdot\langle v\otimes v g_{e}\rangle\big)_x. \label{eq:reformulation_mom_2}
\end{align}
By rescaling the stiff pressure and electric forces through $M^2 \partial_x L$, the auxiliary variable $L$ acts as an effective macroscopic potential. For any $M > 0$, system \eqref{eq:reformulation_mom_1}--\eqref{eq:reformulation_mom_2} is strictly identical to the original momentum equation. In the limit $M \to 0$, while $M^2 \partial_x L \to 0$ recovers the spatial Boltzmann equilibrium constraint, the gradient $\partial_x L$ itself remains finite and captures the limit force required to drive the parallel momentum evolution.

This leads to the augmented micro--macro formulation, whose key properties and asymptotic behavior are formalized in the following proposition.

\begin{proposition}[Augmented micro--macro formulation and \Del{uniform non-degeneracy}] \label{prop2.4}
In an isothermal 1D-1V setting, the augmented 1D-1V micro--macro system for $(n_e, u_e, L, g_e, \phi)$ can be written as:
\begin{align}
    &\partial^2_{tt} n_e - \partial^2_{xx} \big( n_e u_e^2 + L \big) = 0, \label{eq:aug:wave}\\
    &\partial_t (n_e u_e) + \partial_x (n_e u_e^2) + \partial_x L = 0, \label{eq:aug:2}\\
    &M^2 \partial_{x} L = T_{e,0} \partial_{x} n_e - \eta n_e \partial_x \phi + \partial_{x} \langle v^2 g_e \rangle, \label{eq:aug:3}\\
    &M \partial_t g_e + \frac{M}{\kappa} \nu_{ee} g_e = -(\mathbb{I} - \Pi_{\mathcal{M}})\big( v \partial_x g_e \big) + \eta E_x \partial_{v} g_e - (\mathbb{I} - \Pi_{\mathcal{M}})\big( v\partial_x \mathcal{M}_e \big), \label{eq:aug:4}\\
    &-\lambda^2 \eta \partial_{xx} \phi = n_i - n_e. \label{eq:aug:5}
\end{align}
To account for the second-order time derivative in \eqref{eq:aug:wave}, the initial data $(n_e^0, u_e^0)$ are supplemented with the initial density time-derivative compatibility condition:
\begin{equation} \label{eq:aug:compat}
    \left. \partial_t n_e \right|_{t=0} = -\partial_x \big( n_e^0 u_e^0 \big).
\end{equation}

Under condition \eqref{eq:aug:compat} and the assumption $n_e \ge n_\ast > 0$, the following properties hold:
\begin{enumerate}
    \item \emph{Equivalence for $M > 0$:} For any fixed $M > 0$ and $\lambda > 0$, system \eqref{eq:aug:wave}--\eqref{eq:aug:compat} is strictly equivalent to the original 1D-1V micro--macro model.
    \item \emph{Low-Mach limit $M \to 0$:} Under the scaling regime $\kappa = o(M^2)$, as $M \to 0$, the system formally reduces to the limit system established in Proposition~\ref{prop2.3}, where the limit force gradient $\partial_x L_0 = \lim_{M \to 0} \partial_x L$ satisfies:
    \begin{equation} \label{eq:L0_grad_def}
        \partial_x L_0 = T_{e,0} \partial_x n_{e,1} - \eta \big( n_{e,0} \partial_x \phi_1 + n_{e,1} \partial_x \phi_0 \big).
    \end{equation}
    \item \Del{\emph{Uniform structural non-degeneracy as $(M, \lambda) \to (0, 0)$:}  the augmented formulation remains non-singular for all $M \ge 0$ and $\lambda \ge 0$. Each unknown stays governed by a non-degenerate evolution or elliptic equation, uniformly in $M$ and $\lambda$}.
\end{enumerate}
\end{proposition}
The detailed proof of Proposition~\ref{prop2.4} is provided in Appendix~\ref{app:proof}.

%%%-------------------------------------------------------------------------

\section{Numerical scheme}
\subsection{Preparation}
For clarity, we restrict the discretization to a one-dimensional setting in both physical and velocity space along the magnetic-field direction and neglect the magnetic term. The governing equations reduce to
\begin{align}
    &\partial_{t}{n_e}+\partial_{x}(n_e u_e)=0,\label{12}\\
    &M^{2}(\partial_{t}(n_e u_e)+\partial_{x}(n_e u_e^2))+T_{e,0}\partial_{x}n_e-\eta n_e\partial_{x}\phi= -\partial_{x}\langle v^2  g_{e}\rangle,\label{13}\\
    %&M\kappa\partial_t{g_e}+\kappa(\mathbb{I}-\Pi_{\mathcal{M}})(v\partial_xg_e)+\kappa\eta \partial_x\phi\cdot\partial_vg_e=-M\nu_{ee}g_e-\kappa(\mathbb{I}-\Pi_{\mathcal{M}})(v\partial_x\mathcal{M}_{e}),\label{14}\\
    &\Del{M\kappa\partial_t{g_e}+\kappa(\mathbb{I}-\Pi_{\mathcal{M}})\big(v\partial_x(g_e+\mathcal{M}_{e})\big)+\kappa\eta \partial_x\phi\cdot\partial_vg_e=-M\nu_{ee}g_e%-\kappa(\mathbb{I}-\Pi_{\mathcal{M}})(v\partial_x\mathcal{M}_{e})
    ,}\label{14}\\
    &-\lambda^{2}\eta\partial_{xx}\phi=n_{i}-n_e.\label{15}
\end{align}
%The microscopic equation \eqref{14} is recast into
%\Del{\begin{equation}\label{16}
%\begin{aligned}
%    \kappa\partial_t g_e+\nu_{ee}g_e
%    =
%    -\frac{\kappa}{M} \Big( \eta\partial_x\phi\cdot\partial_v g_e +
%    (\mathbb{I}-\Pi_{\mathcal{M}})(v\partial_xg_e) + (\mathbb{I}-\Pi_{\mathcal{M}})(v\partial_x\mathcal{M}_e) \Big),
%\end{aligned}
%\end{equation}}
%where, substituting the explicit expression of the projection operator \eqref{10} into \eqref{14} and using the Maxwellian form above, we obtain
%\begin{align*}
%    (\mathbb{I}-\Pi_{\mathcal{M}})(v\partial_xg)
%    &= v\partial_xg + \frac{1}{n_e}\left(Mu_e\frac{(v-Mu_e)}{T_{e,0}} - 1\right)\langle v\partial_x g\rangle\mathcal{M}_e \\
%    &\quad - \frac{(v-Mu_e)}{n_eT_{e,0}}\langle v^2\partial_xg\rangle\mathcal{M}_e,
%\end{align*}
%\Del{\begin{align*}
%    (\mathbb{I}-\Pi_{\mathcal{M}})(f) = f + \frac{1}{n_e}\left(Mu_e\frac{(v-Mu_e)}{T_{e,0}} - 1\right)\langle f \rangle\mathcal{M}_e - \frac{(v-Mu_e)}{n_eT_{e,0}}\langle v f\rangle\mathcal{M}_e,
%\end{align*}}
%
%We now discretize the micro--macro system.
%
\subsection{Time discretization}

Let $\Delta t$ be the time step, $t^k=k\Delta t$, $k\in\mathbb{N}$,
and denote by $g_e^k$, $n_e^k$, $u_e^k$, $\phi^k$, and
$\mathcal{M}_e^k$ the corresponding approximations at time $t^k$.

\subsubsection{Time discretization of the microscopic equation}

Applying an explicit time discretization to \eqref{14} yields
\Del{\begin{equation}\label{17}
%\begin{multlined}[0.9\textwidth]
    \left(\frac{\kappa}{\Delta t}+\nu_{ee}\right)g_e^{k+1}
    = \frac{\kappa}{\Delta t}g_e^k -\frac{\kappa}{M}\Big( \eta
    \partial_x\phi^k\cdot\partial_v g_e^k +(\mathbb{I}-\Pi_{\mathcal{M}^k})(v\partial_xg_e^k+v\partial_x\mathcal{M}_e^k) \Big),
%\end{multlined}
\end{equation}
where, substituting the explicit expression of the projection operator \eqref{10} into \eqref{14} and using the Maxwellian definition \eqref{eq:Maxwellian:adim}, we obtain
\begin{equation}\label{17:Bis}
%\begin{aligned}
(\mathbb{I}-\Pi_{\mathcal{M}^k})(f^k)=
        f^k
        +\frac{\mathcal{M}_e^k}{n_e^k}
        \left(
            Mu_e^k\frac{v-Mu_e^k}{T_{e,0}}-1
        \right)
        \langle f^k\rangle
        -\frac{v-Mu_e^k}{n_e^kT_{e,0}}
        \langle v \,f^k\rangle\mathcal{M}_e^k.
%\end{aligned}
\end{equation}}

\subsubsection{Time discretization of the macroscopic equations}
Following \cite{yang2026multiscale}, the macroscopic equations \eqref{12}, \eqref{13}, and the Poisson equation \eqref{15} are discretized semi-implicitly as
\begin{align}
    &n_e^{k+1}=n_e^k-\Delta t\, \partial_x(n_e u_e)^{k+1},\label{18}\\
    &(n_e u_e)^{k+1}=(n_e u_e)^k-\Delta t \partial_x(n_e u_e^2)^k -\Delta t T_{e,0}\partial_xn_e^k\label{19}\\
    & \hspace{2cm}- \frac{\Delta t}{M^2}\left((1-M^2)T_{e,0}\partial_xn_e^{k+1}-\eta n_e^k\partial_x\phi^{k+1}\right) -\frac{\Delta t}{M^2}\partial_x\langle v^2g_e^{k+1} \rangle,\notag\\
    &-\lambda^2\eta\partial_{xx}\phi^{k+1}=n_i-n_e^{k+1}.\label{20}
\end{align}
Several comments on this discretization are in order. To address quasi-neutral stiffness, the Poisson equation \eqref{20}, the density flux in \eqref{18}, and the electric force in \eqref{19} are treated implicitly. A fully implicit discretization of $\eta n_e\partial_x\phi$ would introduce a nonlinear term, whereas the semi-implicit approximation $\eta n_e^k\partial_x\phi^{k+1}$ keeps the system linear with acceptable numerical error. Since the micro--macro system is solved in a coupled manner, the microscopic perturbation term in \eqref{19} is also treated implicitly. Finally, following the all-speed semi-implicit strategy proposed in \cite{degond2011all}, the stiff pressure contribution in \eqref{19} is decomposed into an explicit part and an implicit part. The explicit component provides sufficient numerical diffusion to suppress spurious oscillations, whereas the implicit component ensures stability and preserves the correct asymptotic behavior in the low-Mach-number limit.

This semi-implicit treatment is consistent with the philosophy of IMEX schemes for stiff relaxation systems. For hyperbolic equations with stiff sources or relaxation, IMEX methods typically treat non-stiff transport explicitly and stiff collision, pressure, or source terms implicitly, improving stability while maintaining efficiency \cite{jin1996numerical,pareschi2005implicit}. In the present problem, the low-Mach-number limit creates a strong balance between the pressure gradient and the electric force in the macroscopic momentum equation. Standard explicit or ordinary semi-implicit discretizations may amplify discretization and linear-solver errors by small-Mach-number factors, thereby destroying stability and asymptotic consistency. We therefore introduce an auxiliary variable to reconstruct the stiff low-Mach-number balance at the discrete level and obtain an AP scheme that converges to the correct limiting model.

\subsubsection{Asymptotic-preserving scheme}
%We introduce the auxiliary variable
%\begin{equation} \label{L-def}
%    \partial_xL^{k+1}=-\frac{1}{M^2}\left((1-M^2)T_{e,0}\partial_{x}n_e^{k+1}-\eta n_e^k\partial_x\phi^{k+1} + \partial_x\langle v^2g_e^{k+1}\rangle\right) - T_{e,0}\partial_{x}n_e^{k}.
%\end{equation}
\Del{Introducing the auxiliary variable $L$, the} semi-discrete AP scheme is defined as follows.

\begin{definition}[\Del{Asymptotic-preserving time discretization}]\label{def:AP}
\Del{The AP time discretization reads}
\begin{align}
        &(\Delta t)^{-2}\left(n_e^{k+1}-n_e^k \right)+(\Delta t)^{-1}\partial_x(n_e u_e)^k-\partial_{xx}(n_e u_e^2)^k=-\partial_{x\Del{x}}L^{k+1},\label{21}\\
        &(n_eu_e)^{k+1}=(n_eu_e)^k-\Delta t\partial_x(n_eu_e^2)^k+\Delta t\partial_xL^{k+1},\label{22}\\
        &M^2\partial_{x\Del{x}}L^{k+1}=-(1-M^2)T_{e,0}\partial_{x\Del{x}}n_e^{k+1}-M^2T_{e,0}\partial_{x\Del{x}}n_e^{k}\label{23}\\
        &\hspace{5.5cm}+\eta\partial_x(n_e^k\partial_x\phi^{k+1})-\partial_{x\Del{x}}\langle v^2g_e^{k+1} \rangle, \notag\\
        &-\lambda^2\eta\partial_{xx}\phi^{k+1}=n_i-n_e^{k+1},\label{AP-Poi}
\end{align}
\Del{coupled with \eqref{17}--\eqref{17:Bis} for the update of $g_e^{k+1}$.}
\end{definition}
% \textcolor{red}{Use the same notation for the double derivative operators all over the document}
\Del{As $M \to 0$, the original momentum equation degenerates into a static constraint from which the parallel momentum $n_e u_e$ has dropped out, so that the limit system loses a closure equation. The AP time discretization, by contrast, enjoys a uniform non-degeneracy property, retaining an equation from which $n_e u_e$ can be computed for all $M \ge 0$.}

\begin{del}
\begin{proposition}[Uniform non-degeneracy of the AP time discretization in the coupled fluid--low-Mach and quasi-neutral limits]
\label{prop-semidis-lowmach_model}
Under the scaling $\kappa = o(M^2)$, as $M \to 0$, the AP scheme \eqref{21}--\eqref{AP-Poi}
formally reduces to
\begin{align}
    &(\Delta t)^{-2}\big(n_e^{k+1}-n_e^k\big)
      + (\Delta t)^{-1}\partial_x(n_e u_e)^k
      - \partial_{xx}(n_e u_e^2)^k = -\partial_{xx} L^{k+1},\label{21:0} \\
    &(n_e u_e)^{k+1} = (n_e u_e)^k
      - \Delta t\,\partial_x(n_e u_e^2)^k
      + \Delta t\,\partial_x L^{k+1},\label{22:0} \\
    &T_{e,0}\,\partial_{xx} n_e^{k+1} = \eta\,\partial_x(n_e^k \partial_x \phi^{k+1}), \label{23:0} \\
    &-\lambda^2 \eta\,\partial_{xx} \phi^{k+1} = n_i - n_e^{k+1}, \label{24:0}\\
    &g_e^{k+1} = 0,\label{25:0}
\end{align}
a consistent semi-discretization of the low-Mach-number limit system \eqref{eq:M0:1}--\eqref{eq:M0:5b} of Proposition~\ref{prop2.3}. This limit system is non-degenerate, uniformly in the Debye length $\lambda \ge 0$.
\end{proposition}
This proposition shows that the limit model issued from the AP semi-discretization is not degenerate in either the quasi-neutral or the coupled fluid--low-Mach limit. The momentum update \eqref{22:0} keeps its dynamic character, so that $(n_e u_e)^{k+1}$ is advanced directly from the previous step, while the auxiliary variable $L^{k+1}$ is recovered from the density equation \eqref{21:0}, which is left unchanged by the limit. The stiff pressure balance \eqref{23} reduces to \eqref{23:0}, an elliptic equation for the electric potential; it is the differentiated Boltzmann relation $T_{e,0}\,\partial_x n_e^{k+1} = \eta\, n_e^{k+1}\,\partial_x \phi^{k+1}$, recovered up to an $\mathcal{O}(\Delta t)$ time lag in the density. The density and potential are then obtained according to the Debye regime: for $\lambda > 0$, the potential $\phi^{k+1}$ solves Poisson equation \eqref{24:0} and the density $n_e^{k+1}$ follows from \eqref{23:0}; for $\lambda = 0$, equation \eqref{24:0} degenerates into the quasi-neutral constraint $n_e^{k+1} = n_i$, which fixes the density directly, while $\phi^{k+1}$ is retrieved from the elliptic balance \eqref{23:0}. In both cases the update is well defined, so that the AP scheme reduces to the correct limit system, uniformly in $\lambda$, as $M \to 0$.

\begin{proof}[Proof of Proposition~\ref{prop-semidis-lowmach_model}]
The argument follows the same steps as the proof of Proposition~\ref{prop2.3}, now applied at the semi-discrete level. The only difference with the continuous case lies in the density equation, discretized here as the first-order system \eqref{21}--\eqref{22} rather than as the second-order wave equation of Proposition~\ref{prop2.4}: the implicit time coupling between \eqref{21} and \eqref{22} already enforces their consistency, so no additional compatibility condition is required. Non-degeneracy then follows exactly as in Proposition~\ref{prop2.3}, each field remaining governed by a non-singular evolution or elliptic equation.
\end{proof}
\end{del}

%\begin{proof}[Proof of Propositon~\ref{prop-semidis-lowmach_model}]
%Expanding $n_e^k$, $u_e^k$, $\phi^k$, and $L^k$ in powers of $M$, we write
%\begin{align*}
%        n_e^k = \sum_{i=0}^\infty M^i n_{e,i}^k,\quad u_e^k = \sum_{i=0}^\infty M^i u_{e,i}^k,\quad \phi^k = \sum_{i=0}^\infty M^i \phi_{i}^k,\quad L^k = \sum_{i=0}^\infty M^i L^k_i.
%\end{align*}
%From \eqref{L-def}, the zeroth- and first-order Boltzmann balances are
%\begin{align*}
%        &T_{e,0}\partial_xn_{e,0}^{k+1} - \eta n_{e,0}^k\partial_x\phi_0^{k+1} = 0,\\
%        &T_{e,0}\partial_x n_{e,1}^{k+1} - \eta(n_{e,1}^k\partial_x\phi_0^{k+1}+n_{e,0}\partial_x\phi_1^{k+1}) = 0.
%\end{align*}
%Substituting these balances into \eqref{L-def} gives
%\begin{align*}
%        \partial_xL^{k+1} = \eta\sum_{i=0}^2(n_{e,i}^k\partial_x\phi_{2-i}^{k+1})-T_{e,0}\partial_xn_{e,2}^{k+1} + T_{e,0}\partial_xn_e^{k+1} - T_{e,0}\partial_xn_e^{k}.
%\end{align*}
%Inserting this expression into the semi-discrete momentum equation \eqref{22} gives a semi-discrete form of the continuous low-Mach-number limiting model in Proposition~\ref{prop2.3}.
%\end{proof}

%\begin{corollary}
%The micro--macro system equipped with the AP scheme \eqref{21}--\eqref{AP-Poi} correctly degenerates to the semi-discrete low-Mach-number limiting model as $\Liu{\kappa=o(M^2)\to0}$.
%\end{corollary}
%---------------

\subsection{Fully discrete scheme}

Let $x_i=i\Delta x$, $i=0,\ldots,N_x-1$, and
$v_j=v_{\min}+j\Delta v$, $j=0,\ldots,N_v-1$, with
$v_{\min}=-v_{\max}$. The microscopic equation is discretized on the Cartesian phase-space grid, whereas a staggered grid is used for the macroscopic variables: $n_e$,
$\phi$, and $L$ are located at the cell centers $x_i$, while the electron
momentum $(n_eu_e)_{i+1/2}$ is defined at $x_{i+1/2}$. Velocity moments are approximated using the trapezoidal quadrature rule.

We introduce the discrete divergence, gradient, and Laplacian operators
\begin{equation*}
    (D_h\boldsymbol{q})_i
    =
    \frac{q_{i+1/2}-q_{i-1/2}}{\Delta x},
    \quad
    (G_h\boldsymbol{\psi})_{i+1/2}
    =
    \frac{\psi_{i+1}-\psi_i}{\Delta x},
    \quad
    \Delta_h=D_hG_h.
\end{equation*}
The variable-coefficient elliptic operator associated with the electric force and \Del{providing a centered approximation of $\partial_x(n_e^k\partial_x\,\cdot\,)$} is denoted by
\begin{equation*}
    \mathcal{A}_h^k\boldsymbol{\psi}
    =
    D_h\left(\boldsymbol{n}_{e,f}^kG_h\boldsymbol{\psi}\right),
    \quad
    n_{e,i+1/2}^k
    =
    \frac{n_{e,i}^k+n_{e,i+1}^k}{2}.
\end{equation*}
\subsubsection{Microscopic equation}

The phase-space transport terms in \eqref{17} are approximated by
first-order Lax--Friedrichs fluxes. Denoting the corresponding discrete
operators in the physical and velocity directions by
$\mathcal{D}_{x,h}^{\rm LF}$ and $\mathcal{D}_{v,h}^{\rm LF}$,
respectively, the microscopic update is written compactly as
\begin{equation}
\begin{aligned}
    \left(\frac{\kappa}{\Delta t}+\nu_{ee}\right)
    \boldsymbol{g}_e^{k+1}
    &=
    \frac{\kappa}{\Delta t}\boldsymbol{g}_e^k
    -\frac{\kappa}{M}
    \mathcal{D}_{v,h}^{\rm LF}
    \left(\boldsymbol{E}^k\boldsymbol{g}_e^k\right)
    \\
    &\quad
    -\frac{\kappa}{M}
    \left(\mathbb{I}-\Pi_{\mathcal{M},h}^k\right)
    \mathcal{D}_{x,h}^{\rm LF}
    \left[
        v\left(\boldsymbol{g}_e^k+
        \boldsymbol{\mathcal{M}}_e^k\right)
    \right],
\end{aligned}
\label{eq:fully-discrete-micro}
\end{equation}
where $\boldsymbol{E}^k$ denotes the discrete electric force and
$\Pi_{\mathcal{M},h}^k$ is the quadrature-based approximation of the
projection operator \Del{defined by equations \eqref{17}--\eqref{17:Bis}}.

\Del{The following proposition shows that the discrete microscopic component
remains asymptotically orthogonal to the collision invariants associated
with mass and momentum.}

\begin{proposition}[\Del{Asymptotic-Orthogonality of the discrete microscopic component}]
\label{prop_ge_fulldis}
\Del{Let $\boldsymbol{g}_e^{k+1}$ be defined by
\eqref{eq:fully-discrete-micro}. Assume that the initial microscopic
component satisfies
\begin{equation*}
    \left\langle g_{e,i,j}^{0}\right\rangle
    =
    \mathcal{O}(\Delta v+\varepsilon_V),
    \quad
    \left\langle v_jg_{e,i,j}^{0}\right\rangle
    =
    \mathcal{O}(\Delta v+\varepsilon_V),
\end{equation*}
where $\varepsilon_V$ denotes the truncation error associated with the
finite velocity domain. Under the consistency and stability assumptions
on the numerical fluxes, for every fixed time index $k$,}
\begin{equation}
    \max_i
    \left(
        \left|
            \left\langle g_{e,i,j}^{k}\right\rangle
        \right|
        +
        \left|
            \left\langle v_jg_{e,i,j}^{k}\right\rangle
        \right|
    \right)
    \leq
    C_k\left(\Delta x+\Delta v+\varepsilon_V\right),
    \label{eq:discrete-moment-estimate}
\end{equation}
where $C_k$ is independent of $\Delta x$, $\Delta v$, and
$\varepsilon_V$. Consequently,
\begin{equation*}
    \left\langle g_{e,i,j}^{k}\right\rangle\longrightarrow0,
    \quad
    \left\langle v_jg_{e,i,j}^{k}\right\rangle\longrightarrow0,
\end{equation*}
as $\Delta x,\Delta v\to0$ and the truncated velocity domain is enlarged.
\end{proposition}

\Del{This property is central to the consistency of the micro--macro discretization. By construction, the decomposition $f_e = \mathcal{M}_e + g_e$ assigns the mass and momentum of the distribution entirely to the macroscopic component, so that the microscopic part $g_e$ must carry no contribution to these two moments. The
proposition shows that the fully discrete scheme keeps this leakage at the order of discretization parameters and the phase-space truncation, so that the discrete macroscopic moments remain those of $\mathcal{M}_e$ up to $\mathcal{O}(\Delta x + \Delta v +\varepsilon_V)$.}

\begin{proof}
Set
\begin{equation*}
    m_{0,i}^{k}
    =
    \left\langle g_{e,i,j}^{k}\right\rangle,
    \quad
    m_{1,i}^{k}
    =
    \left\langle v_jg_{e,i,j}^{k}\right\rangle.
\end{equation*}
By the orthogonality of
$\mathbb{I}-\Pi_{\mathcal{M},h}^{k}$ and the consistency of the
quadrature-based projection, the zeroth and first moments of the projected
physical-space transport term in \eqref{eq:fully-discrete-micro} are
bounded by
\Del{$\mathcal{O}(\Delta x+\Delta v+\varepsilon_V)$}.
Moreover, discrete summation by parts in the velocity variable gives
\Del{\begin{equation*}
    \left\langle
        \mathcal{D}_{v,h}^{\rm LF}
        \left(\boldsymbol{E}^{k}\boldsymbol{g}_e^{k}\right)
    \right\rangle
    =
    \mathcal{O}(\Delta v+\varepsilon_V),\quad \left\langle
        v_j\mathcal{D}_{v,h}^{\rm LF}
        \left(\boldsymbol{E}^{k}\boldsymbol{g}_e^{k}\right)
    \right\rangle
    =
    -E_i^k m_{0,i}^{k}
    +
    \mathcal{O}(\Delta v+\varepsilon_V).
\end{equation*}}
%and
%\begin{equation*}
%    \left\langle
%        v_j\mathcal{D}_{v,h}^{\rm LF}
%        \left(\boldsymbol{E}^{k}\boldsymbol{g}_e^{k}\right)
%    \right\rangle
%    =
%    -E_i^k m_{0,i}^{k}
%    +
%    \mathcal{O}(\Delta v+\varepsilon_V).
%\end{equation*}
The remainder terms arise from the velocity quadrature and the truncation
of the velocity domain.

Taking successively the zeroth and first discrete velocity moments of
\eqref{eq:fully-discrete-micro} therefore yields
\Del{\begin{align*}
    \begin{aligned}
        |m_{0,i}^{k+1}|
    &\leq
    \theta |m_{0,i}^{k}|
    +
    C\left(\Delta x+\Delta v+\varepsilon_V\right),\\
    |m_{1,i}^{k+1}|
    &\leq
    \theta |m_{1,i}^{k}|
    +
    C|m_{0,i}^{k}|
    +
    C\left(\Delta x+\Delta v+\varepsilon_V\right),
    \end{aligned}
    & \quad \theta
    =
    \frac{\kappa/\Delta t}
    {\kappa/\Delta t+\nu_{ee}}
    <1.
\end{align*}}
%where
%\begin{equation*}
%    \theta
%    =
%    \frac{\kappa/\Delta t}
%    {\kappa/\Delta t+\nu_{ee}}
%    <1.
%\end{equation*}
The estimate \eqref{eq:discrete-moment-estimate} then follows by induction
on $k$ and from the assumed consistency of the initial projection.
\end{proof}

Since the phase-space transport terms are treated explicitly, the time
step is chosen to satisfy
\begin{equation}
    \frac{\kappa}{M(\kappa+\nu_{ee}\Delta t)}
    \frac{v_{\max}\Delta t}{\Delta x}
    \leq C_{\rm CFL},
    \quad
    \frac{\kappa}{M(\kappa+\nu_{ee}\Delta t)}
    \frac{\|\boldsymbol{E}^k\|_{\infty}\Delta t}{\Delta v}
    \leq C_{\rm CFL}.
    \label{eq:microscopic-CFL}
\end{equation}
The factor $\kappa/(\kappa+\nu_{ee}\Delta t)$ relaxes these restrictions
in the strongly collisional regime, consistently with the rapid damping
of the microscopic perturbation.

\subsubsection{Fully discrete AP scheme}

The explicit macroscopic convection term is discretized using a
first-order Lax--Friedrichs flux with local signal speed
$|u_e|+\sqrt{T_{e,0}}$. Let $\boldsymbol{q}^k$ denote the resulting
face-centered approximation of $\partial_x(n_eu_e^2)^k$, and define
\begin{equation*}
    \boldsymbol{\gamma}^k
    =
    \left(
        \left\langle v_j^2g_{e,i,j}^k\right\rangle
    \right)_i.
\end{equation*}
The known right-hand-side vectors are collected as
\Del{\begin{equation}
\begin{aligned}
    \boldsymbol{b}_1^k
    =
    -\frac{\boldsymbol{n}_e^k}{\Delta t^2}
    +\frac{1}{\Delta t}
    D_h\boldsymbol{(n_eu_e)}^k
    -D_h\boldsymbol{q}^k,\quad
    \boldsymbol{b}_2^k
    =
    -M^2T_{e,0}\Delta_h\boldsymbol{n}_e^k
    -\Delta_h\boldsymbol{\gamma}^{k+1}.
\end{aligned}
\label{eq:fully-discrete-rhs}
\end{equation}}
Using centered differences for the implicit pressure, electric-force, and
Poisson terms, the fully discrete AP formulation takes the \Del{matrix} block form
\begin{equation}
    \begin{bmatrix}
        -\Delta_h
        &
        -\Delta t^{-2}I
        &
        0
        \\[0.2em]
        M^2\Delta_h
        &
        (1-M^2)T_{e,0}\Delta_h
        &
        -\eta\mathcal{A}_h^k
        \\[0.2em]
        0
        &
        I
        &
        -\lambda^2\eta\Delta_h
    \end{bmatrix}
    \begin{bmatrix}
        \boldsymbol{L}^{k+1}\\
        \boldsymbol{n}_e^{k+1}\\
        \boldsymbol{\phi}^{k+1}
    \end{bmatrix}
    =
    \begin{bmatrix}
        \boldsymbol{b}_1^k\\
        \boldsymbol{b}_2^k\\
        \boldsymbol{n}_i
    \end{bmatrix}.
    \label{eq:fully-discrete-block-system}
\end{equation}
Once this system has been solved, the electron momentum is updated through
\begin{equation}
    \boldsymbol{(n_eu_e)}^{k+1}
    =
    \boldsymbol{(n_eu_e)}^k
    -\Delta t\,\boldsymbol{q}^k
    +\Delta t\,G_h\boldsymbol{L}^{k+1}.
    \label{eq:fully-discrete-momentum}
\end{equation}

Thus, only $\boldsymbol{L}^{k+1}$, $\boldsymbol{n}_e^{k+1}$, and
$\boldsymbol{\phi}^{k+1}$ enter the principal coupled solve. Under the
scaling $\kappa=o(M^2)$, the second block row of
\eqref{eq:fully-discrete-block-system} recovers the discrete
low-Mach-number balance as $M\to0$, while
\eqref{eq:fully-discrete-momentum} continues to determine the electron
momentum. The fully discrete formulation therefore preserves the
asymptotic structure established for the time-discrete scheme.
Consequently, the fully discrete AP scheme is defined by the microscopic
update \eqref{eq:fully-discrete-micro}, the augmented block system
\eqref{eq:fully-discrete-block-system}, and the momentum update
\eqref{eq:fully-discrete-momentum}.

\subsubsection{Fully discrete IMEX scheme}

For comparison, the fully discrete IMEX scheme is defined by the
microscopic update \eqref{eq:fully-discrete-micro}, followed by the
density--potential system
\begin{equation}
    \begin{bmatrix}
        \displaystyle
        \frac{M^2}{\Delta t^2}I
        -(1-M^2)T_{e,0}\Delta_h
        &
        \eta\mathcal{A}_h^k
        \\[0.4em]
        I
        &
        -\lambda^2\eta\Delta_h
    \end{bmatrix}
    \begin{bmatrix}
        \boldsymbol{n}_e^{k+1}\\
        \boldsymbol{\phi}^{k+1}
    \end{bmatrix}
    =
    \begin{bmatrix}
        -M^2\boldsymbol{b}_1^k-\boldsymbol{b}_2^k\\
        \boldsymbol{n}_i
    \end{bmatrix}.
    \label{eq:fully-discrete-IMEX-density-potential}
\end{equation}
Once $\boldsymbol{n}_e^{k+1}$ and
$\boldsymbol{\phi}^{k+1}$ have been computed, the electron momentum is
updated separately through
\begin{equation}
\begin{aligned}
    \boldsymbol{(n_eu_e)}^{k+1}
    &=
    \boldsymbol{(n_eu_e)}^k
    -\Delta t\,\boldsymbol{q}^k
    -\Delta t\,T_{e,0}G_h\boldsymbol{n}_e^k
    \\
    &\quad
    -\frac{\Delta t}{M^2}
    \left[
        (1-M^2)T_{e,0}G_h\boldsymbol{n}_e^{k+1}
        -\eta\boldsymbol{n}_{e,f}^k
        G_h\boldsymbol{\phi}^{k+1}
        +G_h\boldsymbol{\gamma}^{k+1}
    \right].
\end{aligned}
\label{eq:fully-discrete-IMEX-momentum}
\end{equation}

\subsubsection{Post-processing strategy for the AP scheme}

The AP-post scheme first computes
$\boldsymbol{n}_e^{k+1}$ and $\boldsymbol{\phi}^{k+1}$
from the same density--potential system
\eqref{eq:fully-discrete-IMEX-density-potential} as the fully discrete
IMEX scheme. It then follows the AP formulation:
$\boldsymbol{L}^{k+1}$ is recovered from the first block row of
\eqref{eq:fully-discrete-block-system}, and the electron momentum
$\boldsymbol{(n_eu_e)}^{k+1}$ is updated through
\eqref{eq:fully-discrete-momentum}.

Since \eqref{eq:fully-discrete-IMEX-density-potential} is obtained by exact block
elimination, the AP-post formulation is algebraically equivalent to the
original fully discrete AP scheme and preserves the same asymptotic
properties. Its main advantage is that the principal coupled problem is
reduced from a $3\times3$ to a $2\times2$ block system, resulting in a
simpler algebraic structure for iterative solution and preconditioning. Related block preconditioning techniques have proved effective for anisotropic elliptic systems arising from AP discretizations \cite{li2024block}, and the reduced formulation provides a suitable structure for developing analogous solvers.

%%-------------------------------------------------------------------------
\section{Numerical experiments}
\subsection{Linear Landau damping}
We first consider the dimensionless Vlasov--BGK--Poisson system on the periodic spatial domain $[0,2\pi/k]$, with initial electron distribution
\begin{equation*}
    f_{e,0}(x,v)=\frac{1}{\sqrt{2\pi}}\exp\left(-\frac{v^2}{2}\right)
    \left(1+\alpha\cos(kx)\right),
    \quad (x,v)\in [0,2\pi/k]\times\mathbb{R},
\end{equation*}
where $k=0.5$ and $\alpha=0.05$. The velocity domain is truncated to $[v_{\min},v_{\max}]=[-6,6]$. A uniform Cartesian grid with $N_x=N_v=512$ is used in phase space. For stability, the time step is chosen as $\Delta t=4\times10^{-4}$. Periodic boundary conditions are imposed in space, and homogeneous Neumann boundary conditions are imposed in velocity.

We compare three discretizations. The first one, denoted by V--P, is a direct first-order upwind explicit discretization of the dimensionless Vlasov--BGK--Poisson system. The second one, denoted by IMEX, is the implicit--explicit micro--macro discretization. The third one, denoted by AP, is the \Fabrice{augmented micro--macro denoted AP.}  Unless otherwise specified, all dimensionless parameters $M$, $\lambda$, $\eta$, and $\kappa$ are set to one. For the V--P model, the initial condition is $f_{e,0}$. For the IMEX and AP models, we take
\begin{equation*}
    \mathcal{M}_{e}(x,v,0)=f_{e,0}(x,v),\quad g_e(x,v,0)=0,
    \quad n_e(x,0)=1+\alpha\cos(kx),\quad u_e(x,0)=0 .
\end{equation*}

Figure~\ref{fig:landau-energy} shows the time evolution of the electric-field norm $\|E(t)\|_{L^2}$ on a logarithmic scale. All three methods reproduce \Fabrice{the damping decay with indistinguishable numerical approximations.}
% \st{The IMEX and AP results are nearly indistinguishable over the entire time interval, which is consistent with the theoretical expectation that the AP scheme reduces to the corresponding IMEX discretization away from the low-Mach-number singular limit. The V--P result follows the same damping trend, with only small discrepancies at later times.}}

%\begin{figure}[H]
%    \centering
%    \includegraphics[width=0.5\textwidth]{landau_electric_energy_legend_lower_left.pdf}
%    \caption{Time evolution of the electric-field norm $\|E(t)\|_{L^2}$ in the Landau damping test.}
%    \label{fig:landau-energy}
%\end{figure}

\Fabrice{Figure~\ref{fig:landau-ne-ue-phi} further displays the spatial profiles of the electron density and mean velocity $n_e,u_e$, and the electric potential $\phi$ at time $t=2 $ and $t=10$.}  At $t=2$, the three methods produce nearly identical spatial distributions. \Fabrice{At $t=10$, the IMEX and AP methods show slight phase and amplitude discrepancies relative to the V--P solution, though the overall error remains small. In this fully resolved regime, the IMEX and AP micro--macro schemes the explicit V--P discretization. Here the discretization resolves every scale of the problem (Debye length, thermal velocity). In this fully resolved setting, both the IMEX and AP micro--macro schemes yield results in close agreement with  the explicit V--P discretization closely.}

\begin{figure}[htbp]
    \centering
  \includegraphics[width=0.43\textwidth]{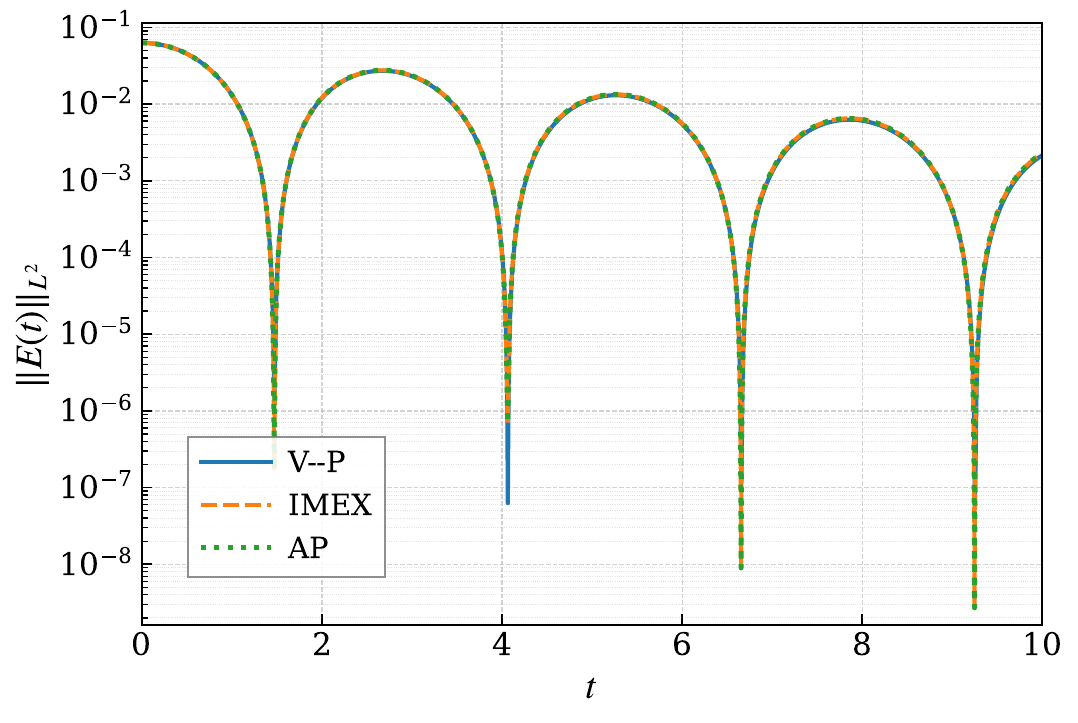}
    \caption{Time evolution of the electric-field norm $\|E(t)\|_{L^2}$ in the Landau damping test.}
    \label{fig:landau-energy}

    \begin{subfigure}[t]{0.33\textwidth}
        \centering
        \includegraphics[width=\textwidth]{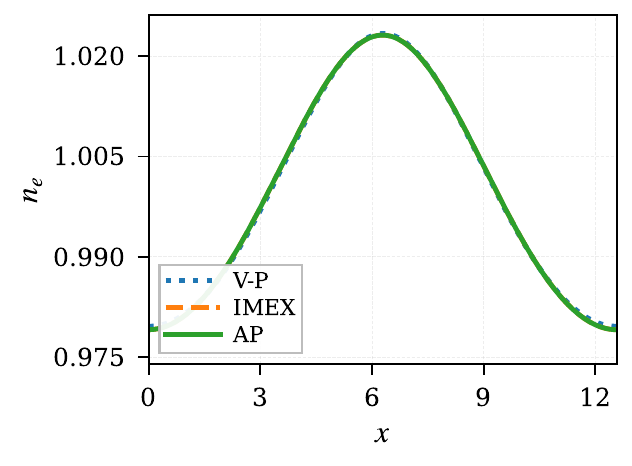}
        \caption{$n_e$ at $t=2$}
        \label{fig:landau-ne-t2}
    \end{subfigure}%
    \begin{subfigure}[t]{0.33\textwidth}
        \centering
        \includegraphics[width=\textwidth]{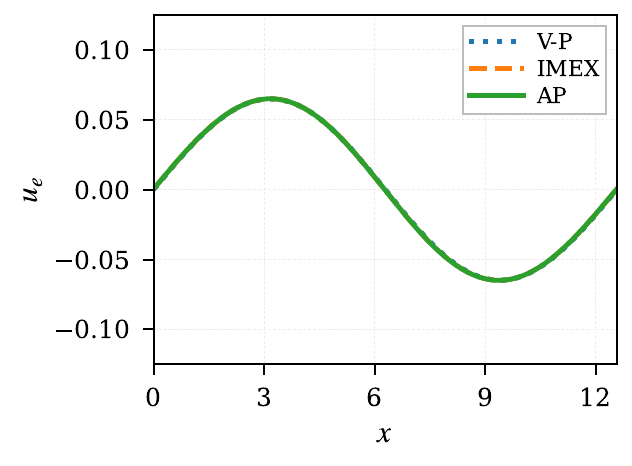}
        \caption{$u_e$ at $t=2$}
        \label{fig:landau-ue-t2}
    \end{subfigure}%
        \begin{subfigure}[t]{0.33\textwidth}
        \centering
        \includegraphics[width=\textwidth]{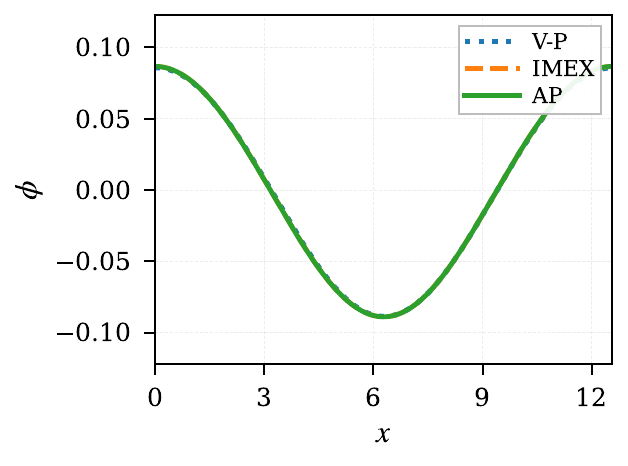}
        \caption{$\phi$ at $t=2$}
        \label{fig:landau-phi-t2}
    \end{subfigure}

    \begin{subfigure}[t]{0.33\textwidth}
        \centering
        \includegraphics[width=\textwidth]{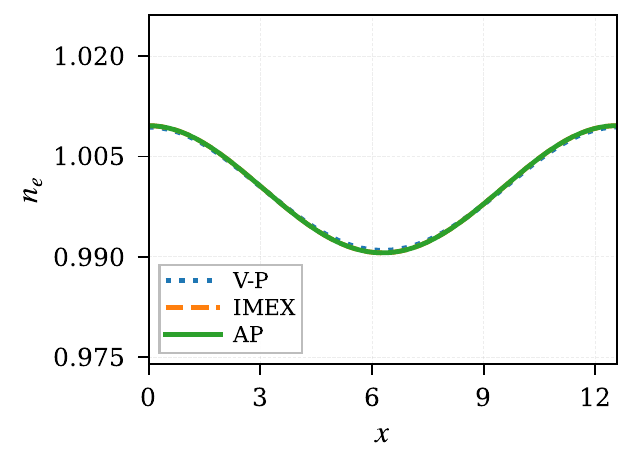}
        \caption{$n_e$ at $t=10$}
        \label{fig:landau-ne-t10}
    \end{subfigure}%
    \begin{subfigure}[t]{0.33\textwidth}
        \centering
        \includegraphics[width=\textwidth]{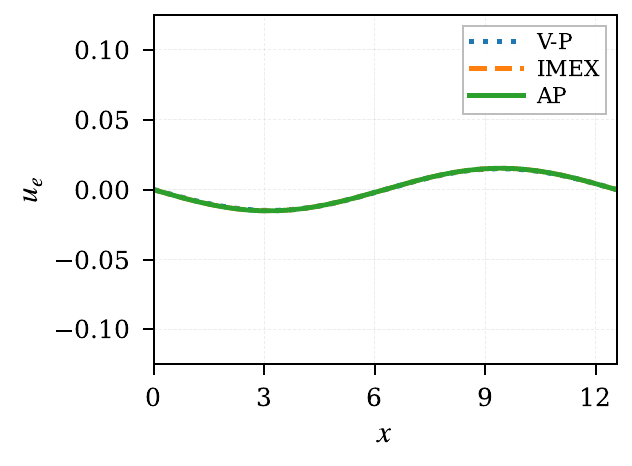}
        \caption{$u_e$ at $t=10$}
        \label{fig:landau-ue-t10}
    \end{subfigure}%
    \begin{subfigure}[t]{0.33\textwidth}
        \centering
        \includegraphics[width=\textwidth]{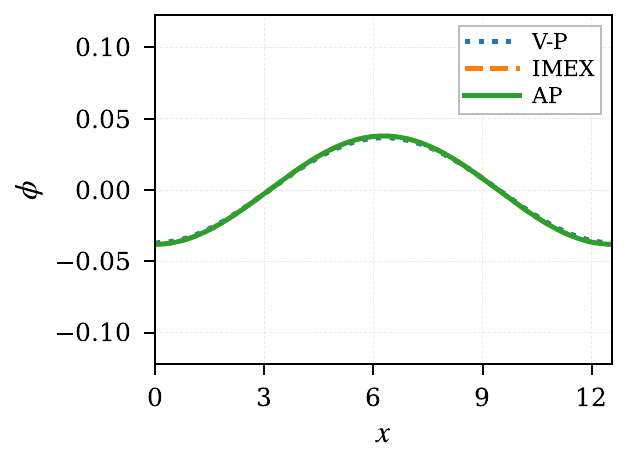}
        \caption{$\phi$ at $t=10$}
        \label{fig:landau-phi-t10}
    \end{subfigure}

    \caption{Spatial profiles of the electron density $n_e$, electron velocity $u_e$, and electric potential $\phi$ in the linear Landau damping test. From top to bottom, the rows correspond to $t=2$ and $t=10$, respectively. From left to right, the columns correspond to $n_e$, $u_e$, and $\phi$, respectively. The V--P, IMEX, and AP results are compared in each subfigure.}
    \label{fig:landau-ne-ue-phi}
\end{figure}
\subsection{Low-Mach-number limit}
\Fabrice{
The low-Mach-number limit places a specific numerical difficulty related to the degeneracy of the continuous model
analyzed in Section~2.4.2. At the discrete level, the stiff force balance in the macroscopic momentum equation carries a factor of order $1/M^2$, so that a linear-solver residual meeting a prescribed (iterative solver) tolerance can still be amplified in the velocity update and accumulate over time. This mechanism is easily overlooked and tends to remain hidden in the literature. It is indeed tied to iterative solvers with a finite stopping tolerance: low-Mach and AP schemes are frequently assessed on one-dimensional problems small enough to be handled by direct methods, whose residuals lie at round-off and therefore never trigger the $1/M^2$ amplification. To isolate it, we deliberately avoid nontrivial physical dynamics and use an exact, time-independent equilibrium, whose continuous solution is known and stationary. Any departure of the computed solution is then attributable solely to the scheme, so that the velocity error $E_u(t)=\left\|u_e(\cdot,t)-1\right\|_{\infty}$ directly measures its ability to preserve the low-Mach-number balance as $M \to 0$. We compare the standard semi-implicit IMEX micro--macro scheme, which serves as the baseline and exhibits the pathology, the proposed AP scheme and AP-post formulation which reconstruct the stiff balance through the auxiliary variable $L$.}

The dimensionless parameters \Fabrice{are $ \lambda=10^{-2}$,  $\eta=1$ with $\kappa=M^2$ and the Mach number is to $M=10^{-4}$ and decreasing to  $10^{-7}$.} The velocity domain is $[-v_{\max},v_{\max}]$ with $v_{\max}=7$. The numbers of grid points in the spatial and velocity directions are $N_x=201$ and $N_v=401$, respectively, and the time step is $\Delta t=10^{-3}$. Except for the Mach number, all numerical parameters are kept fixed in all tests.

\Fabrice{The steady state solution is defined by 
\begin{equation*}
    n_e(x,t)=1,
    \quad u_e(x,t)=1,
    \quad g_e(x,v,t)=0,
    \quad \phi(x,t)=0,
    \quad L(x,t)=0.
\end{equation*}
The exact solution serves as an initial data and left boundary condition for the macroscopic quantities $(n_e,u_e,\phi,L)=(1,1,0,0)$}. 
At the right boundary, homogeneous Neumann conditions are imposed for these macroscopic variables. Homogeneous Neumann conditions are imposed for the microscopic perturbation $g_e$ at both ends of the spatial domain. Since the initial and boundary conditions are consistent with the exact equilibrium, \Fabrice{the numerical solution is expected to} preserve the steady state throughout the computation.

\Fabrice{All linear systems arising at each time step are solved by a preconditioned Krylov method, namely GMRES with an incomplete-LU (ILU) preconditioner, using a drop tolerance $\tau = 10^{-2}$. This iterative choice is deliberate: unlike a direct solver, GMRES returns an approximate solution whose residual is controlled only up to a prescribed stopping tolerance, which is precisely the finite residual whose amplification we wish to probe. To this end, we let the tolerances follow the Mach number, $\varepsilon_{\rm rel} = M,
    \
    \varepsilon_{\rm abs} = M^2, $
% \begin{equation*}
%     \varepsilon_{\rm rel} = M,
%     \quad
%     \varepsilon_{\rm abs} = M^2 ,
% \end{equation*}
so that the solver is required to work increasingly hard as $M \to 0$. Even so, the standard IMEX scheme retains terms of order $1/M^2$ in the low-Mach-number stiff part, so that a residual meeting these tolerances is still amplified by a factor $\sim 1/M$ in the velocity update, an error that a direct solver, returning residuals at round-off, would never reveal for small size one-dimensional problems.}

Figure~\ref{fig:low-mach-mach-sweep} shows the time evolution of the \Fabrice{equilibrium preservation} error $E_u(t)$ for the three schemes and different Mach numbers. The standard IMEX scheme exhibits clear error accumulation in all tests. For $M=10^{-4}$, the IMEX velocity errors at $t=1$ is \Fabrice{larger than $50\%$. For $M=10^{-7}$, the error reaches $2.45\times10^2$, i.e. the computed parallel
velocity departs from its exact value $u_e=1$ by more than two hundred times.} In contrast, the AP scheme keeps the velocity error very small for all Mach numbers; even for $M=10^{-7}$, the error at $t=1$ is only $2.14\times10^{-11}$. The AP-post results further show that the post-processing formulation preserves AP stability and improves the algebraic solver behavior \Fabrice{by reducing the size of the linear system}; its error is generally smaller than that of the original AP scheme.

\begin{figure}[htbp]
    \centering
    \includegraphics[width=0.95\textwidth]{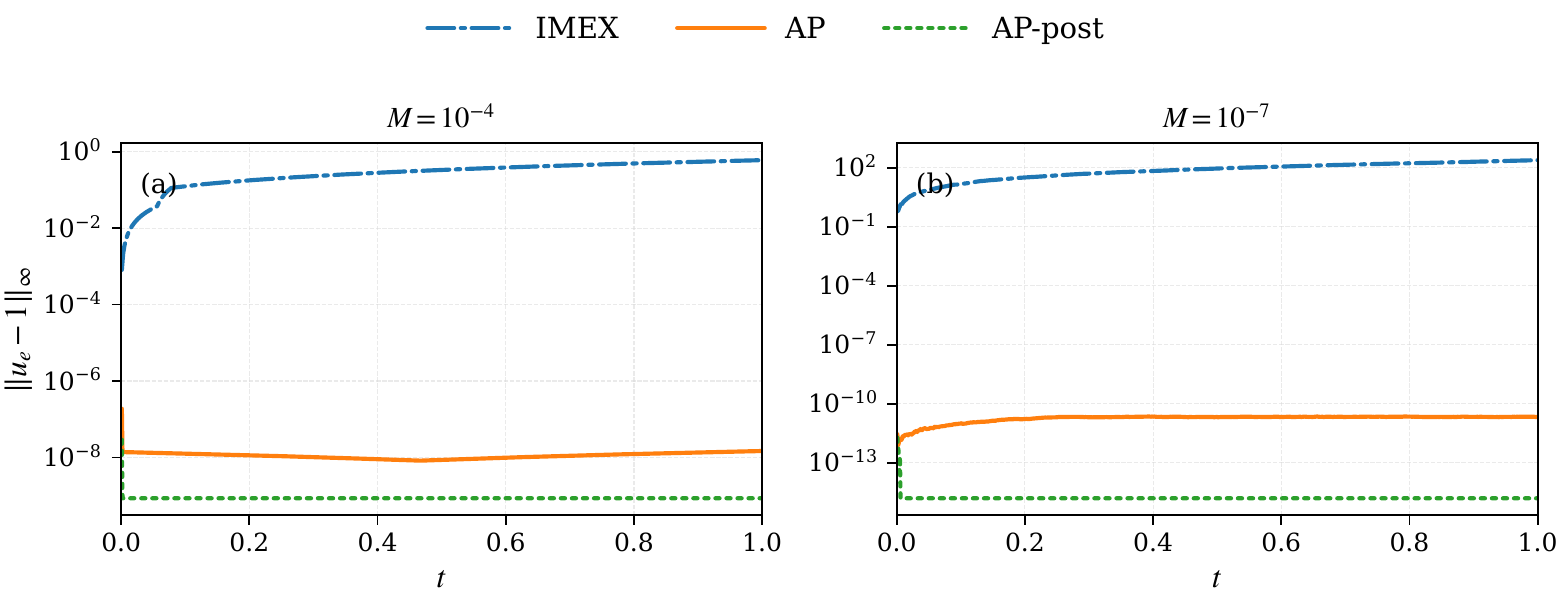}
    \caption{Low-Mach-number equilibrium test: time evolution of the velocity error $E_u(t)=\left\|u_e(\cdot,t)-1\right\|_{\infty}$ computed by the IMEX, AP, and AP-post schemes for different Mach numbers.}
    \label{fig:low-mach-mach-sweep}
\end{figure}

Table~\ref{tab:low-mach-mach-sweep} reports the velocity error at $t=1$ and the maximum error over the whole computational interval. \Fabrice{To quantify the difference among the schemes, we define the error--reduction factor}
\begin{equation*}
    R_{\rm AP}(t)=
    \frac{\left\|u_{e,\rm IMEX}(\cdot,t)-1\right\|_{\infty}}
    {\left\|u_{e,\rm AP}(\cdot,t)-1\right\|_{\infty}},
    \quad
    R_{\rm AP-post}(t)=
    \frac{\left\|u_{e,\rm IMEX}(\cdot,t)-1\right\|_{\infty}}
    {\left\|u_{e,\rm AP-post}(\cdot,t)-1\right\|_{\infty}} .
\end{equation*}

\begin{table}[htbp]
    \centering
    \caption{Low-Mach-number equilibrium test with Mach-dependent GMRES tolerances. Here \Fabrice{$E^1=E_u(t=1)$} denotes the velocity error at $t=1$, and \Fabrice{$E^{\max}=\max_{t\leq 1}E_u(t)$} denotes the maximum error over $0<t\le1$.}
    \label{tab:low-mach-mach-sweep}
    \resizebox{\textwidth}{!}{%
    \begin{tabular}{ccccccccc}
        \toprule
        $M$ & $\varepsilon_{\rm rel}$ & $\varepsilon_{\rm abs}$ & $E^1_{\rm IMEX}$ & $E^1_{\rm AP}$ & $E^1_{\rm AP-post}$ & $E^{\max}_{\rm IMEX}$ & $E^{\max}_{\rm AP}$ & $E^{\max}_{\rm AP-post}$ \\
        \midrule
        $1\times10^{-4}$ & $1\times10^{-4}$ & $1\times10^{-8}$ & $6.02\times10^{-1}$ & $1.47\times10^{-8}$ & $8.55\times10^{-10}$ & $6.02\times10^{-1}$ & $1.88\times10^{-7}$ & $2.92\times10^{-8}$ \\
        $1\times10^{-5}$ & $1\times10^{-5}$ & $1\times10^{-10}$ & $5.43\times10^{-1}$ & $9.26\times10^{-8}$ & $7.93\times10^{-12}$ & $5.43\times10^{-1}$ & $1.35\times10^{-6}$ & $2.91\times10^{-10}$ \\
        $1\times10^{-6}$ & $1\times10^{-6}$ & $1\times10^{-12}$ & $4.67\times10^{-2}$ & $1.13\times10^{-11}$ & $1.81\times10^{-12}$ & $5.36\times10^{-2}$ & $2.17\times10^{-11}$ & $3.63\times10^{-12}$ \\
        $1\times10^{-7}$ & $1\times10^{-7}$ & $1\times10^{-14}$ & $2.45\times10^{2}$ & $2.14\times10^{-11}$ & $1.55\times10^{-15}$ & $2.45\times10^{2}$ & $2.26\times10^{-11}$ & $1.89\times10^{-12}$ \\
        \bottomrule
    \end{tabular}%
    }
\end{table}

\Fabrice{The IMEX scheme fails to improve accuracy, and in fact it deteriorates, as $M \to 0$, despite the tightening of the GMRES tolerance ($\varepsilon_{\rm rel}=M$, $\varepsilon_{\rm abs}=M^2$), revealing a nontrivial coupling between linear-solver errors and low-Mach-number stiffness in the standard discretization. Both AP and AP-post schemes, by contrast, preserve the equilibrium for all Mach numbers tested: at $M=10^{-7}$ and $t=1$, the IMEX error is $2.45\times10^2$, against $2.14\times10^{-11}$ and $1.55\times10^{-15}$
for AP and AP-post. Table~\ref{tab:low-mach-mach-sweep} quantifies this gap
through the error-reduction factors, which at $t=1$ range from $5.86\times10^6$
to $1.14\times10^{13}$ for AP, and reach $1.58\times10^{17}$ for AP-post. The AP
scheme removes the amplification by reconstructing the low-Mach-number balance
at the discrete level, and the AP-post formulation further improves the
algebraic solve while preserving the same asymptotic structure.}

\subsection{Plasma expansion}
We finally consider a one-dimensional plasma expansion into vacuum. Plasma expansion is a classical benchmark for testing the multiscale capability of plasma solvers. The evolution typically involves rapid electron diffusion, local charge separation, formation of a self-consistent sheath electric field, and ion acceleration driven by that field \cite{mora2003plasma}. At the early stage of expansion, the electron thermal velocity is much larger than the ion thermal velocity, so electrons diffuse into the vacuum first and induce local charge separation near the plasma front. The resulting self-consistent electric field accelerates ions and drives the overall plasma expansion. This test therefore contains a quasi-neutral bulk region, a locally non-neutral transition region, fast electron response, and slow ion-scale dynamics.

% Previous studies have shown that electron kinetic effects may influence the expansion front, the electric-field structure, and the ion acceleration process, and therefore a simple isothermal electron-fluid closure may not fully describe the associated thermal and non-equilibrium effects \cite{grismayer2008electron}. In this test, we use a Vlasov description for ions and a micro--macro description for electrons. The electron temperature is recovered from the second velocity moment of the microscopic perturbation, allowing us to assess the ability of the proposed method to describe electron non-equilibrium thermal effects. The numerical results are compared with a PIC reference solution. The quantities compared include the ion density $n_i$, electron density $n_e$, ion macroscopic velocity $u_i$, electric field $E$, and electron temperature $T_e$. Here the electron temperature is not prescribed by an isothermal closure. Instead, with $T_{e,0}$ denoting the background equilibrium temperature, it is computed as
% \[
%     T_e(x,t)
%     =
%     T_{e,0}
%     +
%     \frac{1}{n_e(x,t)}
%     \int_{\mathbb{R}} v^2 g_e(x,v,t)\,dv .
% \]
% This expression shows that spatial variations of the electron temperature are generated by the second velocity moment of the microscopic perturbation $g_e$. Thus the method describes not only the macroscopic electron-density response but also the thermal-state variation caused by deviations from equilibrium.
Previous studies have shown that electron kinetic effects may influence the expansion front, the electric-field structure, and the ion acceleration process, and therefore a simple isothermal electron-fluid closure may not fully describe the associated thermal and non-equilibrium effects \cite{grismayer2008electron}. In this test, we use a Vlasov description for ions and a micro--macro description for electrons. The electron temperature is recovered from the second velocity moment of the microscopic perturbation \Fabrice{using $T_e(x,t)
    =
    T_{e,0}
    +
    \frac{1}{n_e(x,t)}
    \int_{\mathbb{R}} v^2 g_e(x,v,t)\,dv$,} allowing us to assess the ability of the proposed method to describe electron non-equilibrium thermal effects. The numerical results are compared with a PIC reference solution.

The test is nondimensionalized using electron units. The characteristic length, time, velocity, temperature, potential, magnetic field, and number-density scales are chosen as $\bar{x}=10^{-3}$, $\bar{t}=10^{-3}$, $\bar{u}=1$, $\bar{T}=1$, $\bar{\phi}=1$, $\bar{B}=1$, and $\bar{n}=10^6$, respectively. In these units, the electron mass, elementary charge, vacuum permittivity, vacuum permeability, and Boltzmann constant are normalized as $m_e=1$, $e=1$, $\varepsilon_0=1$, $\mu_0=1$, and $k_B=1$. The ion mass is $m_i=1836$, corresponding to a typical proton--electron mass ratio. The electron background temperature and ion temperature are $T_{e,0}=1$ and $T_i=10^{-3}$. Under these parameters, the electron dynamics and ion response scales are clearly separated, and the dimensionless ion sound speed is $c_s=\sqrt{k_BT_{e,0}/m_i}\approx 2.33\times10^{-2}$.

The computational domain is $x\in[0,200]$ with spatial mesh size $\Delta x=0.5$. The left endpoint $x=0$ corresponds to the center of the original symmetric problem, while the right endpoint represents the far vacuum boundary. The electron velocity domain is $v_e\in[-7,7]$ with 401 grid points, giving $\Delta v_e=3.5\times10^{-2}$. The ion velocity domain is $v_i\in[-1,7]$, also with 401 grid points, giving $\Delta v_i=2.0\times10^{-2}$. Each dimensionless time unit is advanced with 200 time steps, so $\Delta t=5.0\times10^{-3}$. Diagnostic outputs are recorded every one dimensionless time unit. Unless otherwise stated, all numerical results are reported in these electron units.

The boundary conditions are consistent with plasma expansion into vacuum. At $x=0$, symmetry is imposed by using reflective boundary conditions for the distribution functions and setting $E(0,t)=0$. At the right boundary, absorbing boundary conditions are imposed for the distribution functions to model particles leaving the computational domain and entering the far-field vacuum without returning. An open boundary treatment is used for the electric field to reduce artificial boundary effects on the interior expansion.

Figure~\ref{fig:plasma-expansion-results} presents the ion density $n_i$, electron density $n_e$, ion macroscopic velocity $u_i$, electric field $E$, and electron temperature $T_e$ at $t=9.8\tau_{pi}$. The density profiles are shown on logarithmic scales to resolve both the plasma front and the low-density tail. For the temperature comparison, the AP solution is assessed against the PIC reference and Vlasov--Poisson solutions obtained with first-, third-, and fifth-order conservative semi-Lagrangian (CSL) methods\cite{yang2021highly}.

\begin{figure}[htbp]
    \centering
    \begin{subfigure}[t]{0.48\textwidth}
        \centering
        \includegraphics[width=\textwidth]{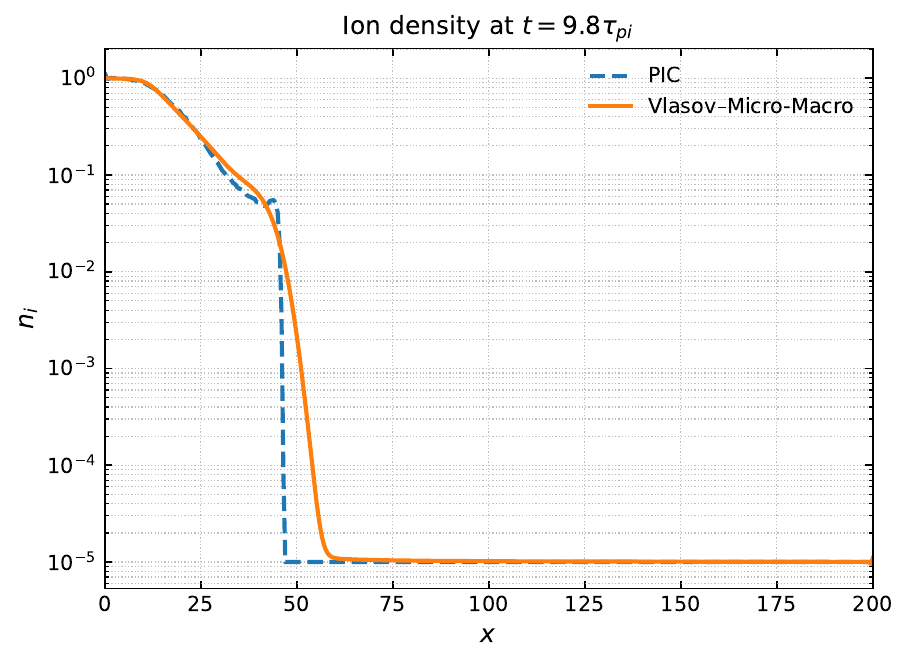}
        \caption{Ion density $n_i$}
        \label{fig:plasma-expansion-ni}
    \end{subfigure}
    \hfill
    \begin{subfigure}[t]{0.48\textwidth}
        \centering
        \includegraphics[width=\textwidth]{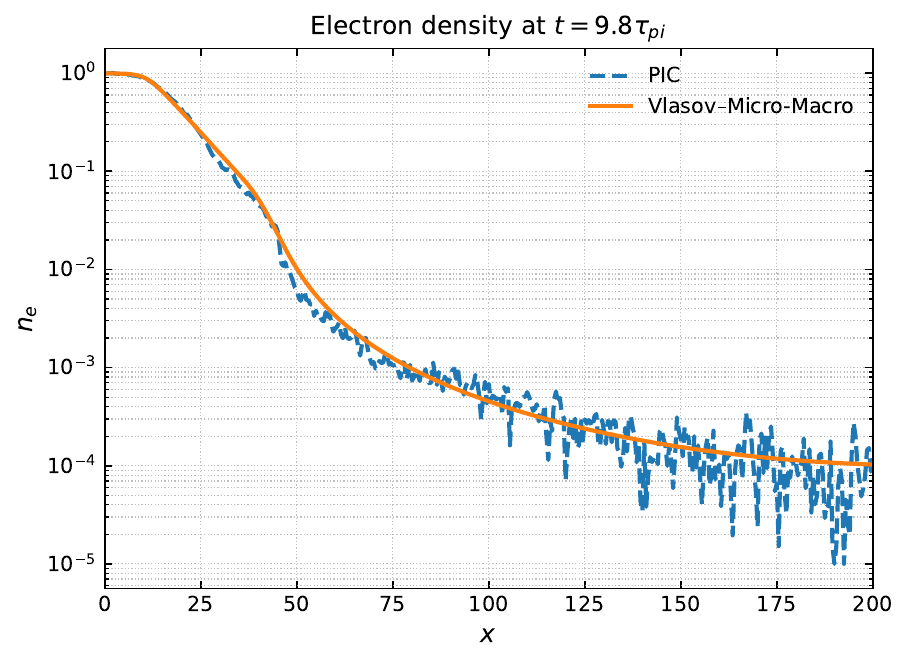}
        \caption{Electron density $n_e$}
        \label{fig:plasma-expansion-ne}
    \end{subfigure}

    \vspace{-0.1em}

    \begin{subfigure}[t]{0.48\textwidth}
        \centering
        \includegraphics[width=\textwidth]{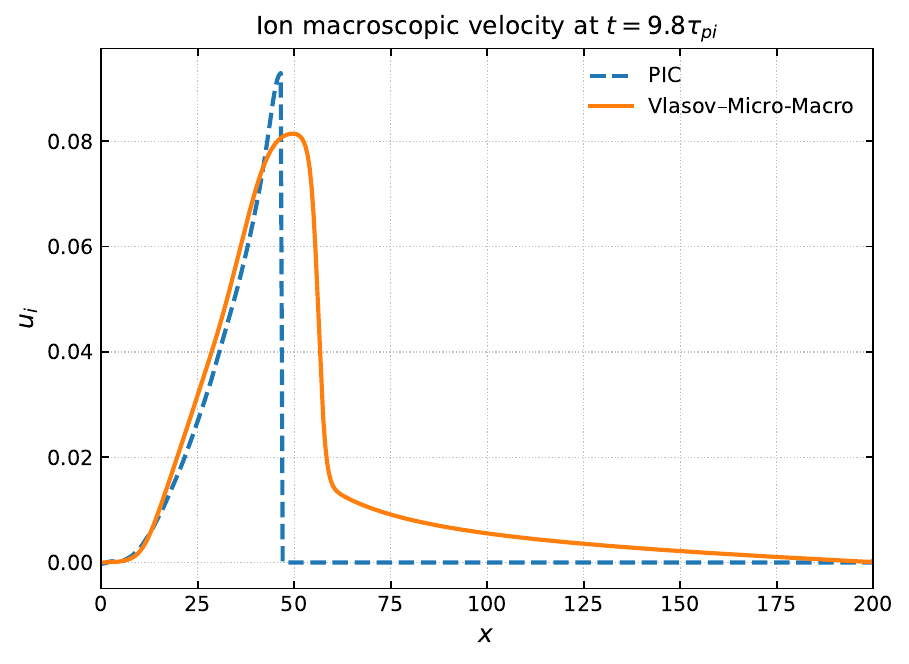}
        \caption{Ion macroscopic velocity $u_i$}
        \label{fig:plasma-expansion-ui}
    \end{subfigure}
    \hfill
    \begin{subfigure}[t]{0.48\textwidth}
        \centering
        \includegraphics[width=\textwidth]{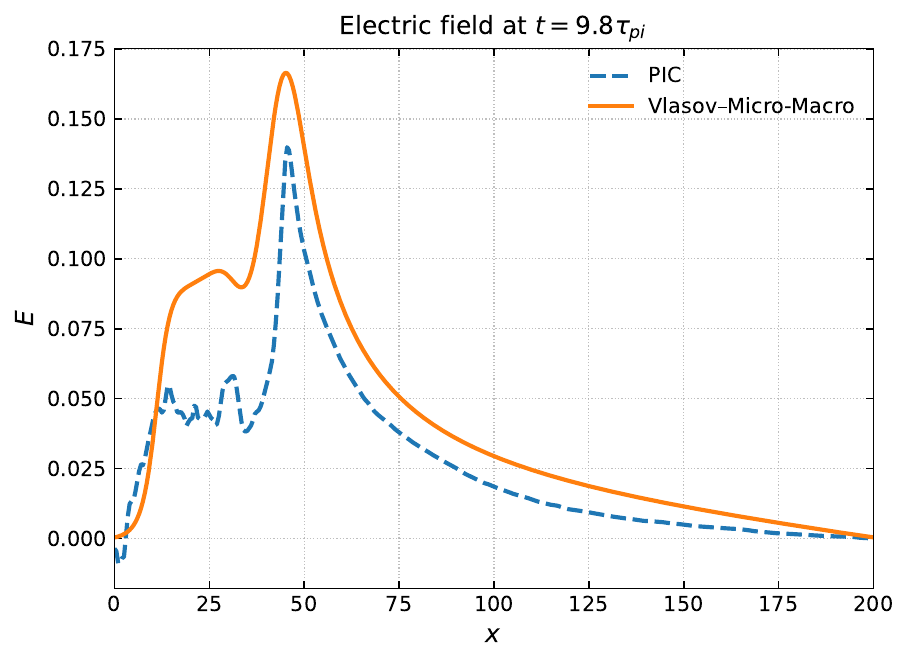}
        \caption{Electric field $E$}
        \label{fig:plasma-expansion-electric-field}
    \end{subfigure}

    \vspace{-0.1em}

    % --- Dernière ligne : sous-figure (e) à gauche, caption à droite ---
    \begin{minipage}[c]{0.48\textwidth}
        \centering
        \begin{subfigure}[t]{\textwidth}
            \centering
            \includegraphics[width=\textwidth]{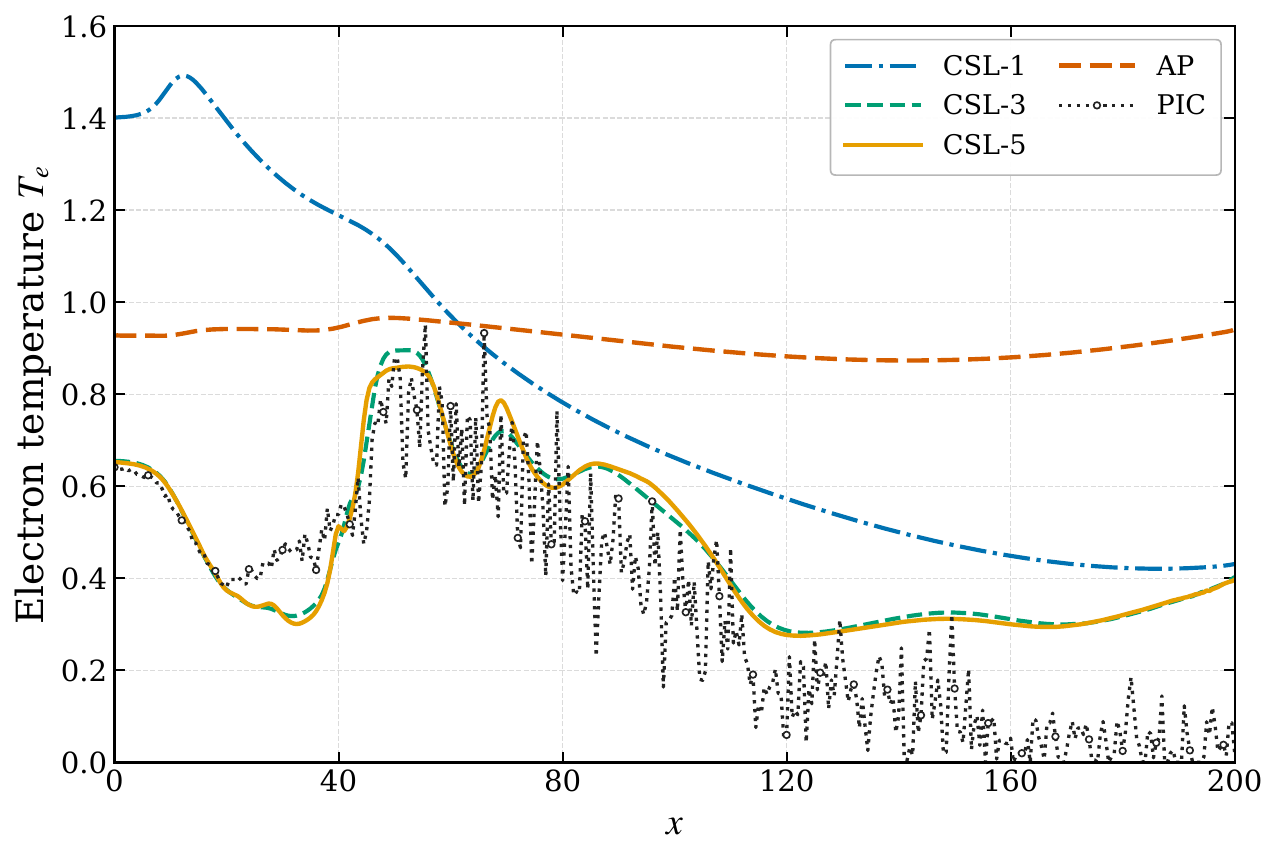}
            \caption{Electron temperature $T_e$}
            \label{fig:plasma-expansion-Te-comparison}
        \end{subfigure}
    \end{minipage}
    \hfill
    \begin{minipage}[c]{0.48\textwidth}
        \caption{\Del{Numerical results for the one-dimensional plasma expansion
        problem at $t=9.8\tau_{pi}$: (a) ion density, (b) electron density,
        (c) ion macroscopic velocity, (d) electric field, and (e) electron
        temperature. The present AP results are compared with the PIC reference
        solution in all panels; panel (e) additionally includes Vlasov--Poisson
        solutions obtained with first-, third-, and fifth-order CSL
        methods. The vertical axes in panels (a) and (b) are logarithmic.}}
        \label{fig:plasma-expansion-results}
    \end{minipage}
\end{figure}

Figures~\ref{fig:plasma-expansion-results}(a)--(d) show that the AP and PIC solutions capture the expansion front near $x\approx50$, the electron low-density tail, the acceleration of ions, and the localized electric field generated in the plasma--vacuum transition region. The AP solution agrees well with PIC in the plasma bulk and predicts the front position correctly. Its density transition and electric-field profile are somewhat broader, while the ion-velocity peak is slightly lower and shifted downstream. These discrepancies are consistent with the numerical diffusion of the first-order spatial discretization.

The temperature comparison in Figure~\ref{fig:plasma-expansion-results}(e) demonstrates a stronger sensitivity to phase-space accuracy. The CSL-1 solution overestimates the bulk temperature and smooths the non-monotone structures near and beyond the expansion front, whereas the CSL-3 and CSL-5 profiles are nearly indistinguishable and resolve the principal cooling and reheating features. The PIC result follows the same large-scale trend but becomes noisy in the low-density region, where the normalized second moment is sensitive to particle statistics. The AP solution remains smooth and avoids the large CSL-1 overestimation, but its weak temperature variation does not reproduce the detailed structures resolved by the higher-order CSL methods.

Overall, the proposed method stably captures the principal macroscopic dynamics of plasma expansion, while the temperature results indicate that an accurate description of higher-order kinetic moments requires a higher-order phase-space discretization within the AP framework.

%%---------------------------------------------------------------
\section{Conclusion and perspectives}

\begin{del}
    An asymptotic-preserving micro--macro framework has been developed for the
Vlasov--BGK--Poisson system in the quasi-neutral and low-Mach-number regimes.
Decomposing the electron distribution into a local Maxwellian and a microscopic
perturbation combines a macroscopic description of the dominant fluid behavior
with the kinetic corrections carried by the non-equilibrium component. The
formulation remains consistent with the original kinetic model for finite values
of the scaling parameters and degenerates into the appropriate macroscopic
balance as the Mach number and the scaled Debye length vanish.

A semi-implicit discretization handles the stiffness of the electric force, the
pressure contribution, and the Poisson coupling. Its distinctive feature is the
auxiliary-variable reformulation, which rescales the stiff force term and thereby
avoids the amplification of algebraic errors by the singular factor $1/M^2$: this
is what makes the AP property hold along the entire solution chain, including an
iterative linear solver with finite tolerance. An algebraically equivalent
post-processing formulation eliminates the auxiliary variable from the principal
solve, reducing the coupled system while preserving the AP property and offering
a simpler structure for preconditioner design.

The numerical experiments span complementary regimes. Landau damping confirms
that the micro--macro formulation retains the relevant kinetic behavior away from
the singular limits; the low-Mach equilibrium test shows that the AP and AP-post
schemes remain accurate at very small Mach numbers even with inexact iterative
solvers, whereas a direct treatment suffers substantial amplification of solver
errors; plasma expansion into vacuum demonstrates that the method captures the
principal dynamics and that the microscopic correction is necessary to recover
thermal effects beyond an isothermal fluid closure. The temperature results also
expose the influence of phase-space discretization, with low-order interpolation
introducing noticeable numerical diffusion.

Future work will address higher-order, less dissipative phase-space
discretizations and a rigorous fully discrete stability and error analysis
accounting for iterative-solver tolerances, together with efficient block
preconditioners for the reduced system. Extensions to multidimensional and
electromagnetic configurations, and adaptive kinetic treatments for strongly
non-equilibrium regions, are natural next steps: plasma sheaths and
plasma--wall interactions are a prime application, where localized departures
from the low-Mach and quasi-neutral regimes must be resolved without
compromising efficiency elsewhere.
\end{del}

%
%%-------------------------------------------------------------
  \appendix
\section{Proofs of the main results}  \label{app:proof}
\begin{proof}[Proof of Lemma~\ref{lem:ge_estimate_1d}]
In the 1D-1V setting, the microscopic equation \eqref{MMk} reads
\begin{equation} \label{eq:ge_governing_1d}
        \begin{gathered}
    M \partial_t g_e + \frac{M}{\kappa}\nu_{ee}\, g_e
    = \mathcal{T}(g_e) + S_e,
    \\
        \mathcal{T}(g) := -(\mathbb{I}-\Pi_{\mathcal{M}})
                            \big(v\,\partial_x g\big)
                            + \eta E\, \partial_v g , \quad
        S_e := -(\mathbb{I}-\Pi_{\mathcal{M}})
                 \big(v\,\partial_x\mathcal{M}_e\big) .
    \end{gathered}
\end{equation}
Multiplying by $\kappa/M$ gathers the collision term and the perturbations on the left-hand side,
\begin{equation} \label{eq:ge_balance}
    \left( \nu_{ee} + \kappa\, \partial_t
           - \frac{\kappa}{M}\,\mathcal{T} \right) g_e
    = \frac{\kappa}{M}\, S_e .
\end{equation}

The rest of the proof is broken into 3 steps.
\begin{enumerate}
    \item We derive an estimate on the source term $S_e$. Differentiating $\mathcal{M}_e$ with respect to $x$ gives
\begin{equation*}
    v\,\partial_x \mathcal{M}_e
    = \frac{\partial_x n_e}{n_e}\, v\,\mathcal{M}_e
    + M\,\frac{\partial_x u_e}{T_{e,0}}\, v\,(v-Mu_e)\,\mathcal{M}_e .
\end{equation*}
The first term is exactly cancelled by $\mathbb{I}-\Pi_{\mathcal{M}}$: indeed,
$v\,\mathcal{M}_e = \big[(v-Mu_e) + Mu_e\big]\mathcal{M}_e$ belongs to
$\operatorname{Im}(\Pi_{\mathcal{M}})$ as a linear combination of the elements
of the basis $\mathcal{B}$ defined in \eqref{9}.

Writing likewise $v(v-Mu_e) = (v-Mu_e)^2 + Mu_e(v-Mu_e)$ and using the moment identities
$\langle (v-Mu_e)^2\mathcal{M}_e\rangle = n_e T_{e,0}$ and
$\langle (v-Mu_e)^3\mathcal{M}_e\rangle = 0$, we obtain
\begin{equation} \label{eq:source_exact}
    S_e = -M\, \Sigma,
    \quad
    \Sigma := \frac{\partial_x u_e}{T_{e,0}}
              \Big( (v-Mu_e)^2 - T_{e,0} \Big)\mathcal{M}_e
            = \mathcal{O}(1) .
\end{equation}
Consequently the right-hand side of \eqref{eq:ge_balance} equals $-\kappa\,\Sigma = \mathcal{O}(\kappa)$. This estimate is instrumental in deriving the optimal bound for $g_e$.
\item We now derive a bound on $g_e$ by successive substitution, in the spirit of
the iterative closure of the Hilbert expansion introduced in
\cite{Caflisch1980}. To this end, equation \eqref{eq:ge_balance} is recast in the implicit form
\begin{equation} \label{eq:ge_implicit}
    g_e = \frac{1}{\nu_{ee}}
          \left( -\kappa\, \partial_t g_e + \frac{\kappa}{M}\,\mathcal{T}(g_e) - \kappa\, \Sigma \right),
\end{equation}
which expresses $g_e$ in terms of itself. Any available bound may therefore be substituted into the right-hand side to produce a new bound on the left-hand side. This is performed at fixed time and does not involve any time marching.

\begin{enumerate}
\item \emph{First pass.} 
Using the $M$-uniform a priori bounds of assumption~(iii)
to estimate $g_e = \mathcal{O}(1)$, $\partial_t g_e = \mathcal{O}(1)$ and
$\mathcal{T}(g_e) = \mathcal{O}(1)$, \eqref{eq:ge_implicit} yields
\begin{equation} \label{eq:first_pass}
    g_e = \mathcal{O}(\kappa) + \mathcal{O}(\kappa)
        + \mathcal{O}\!\left(\frac{\kappa}{M}\right)
        = \mathcal{O}\!\left(\frac{\kappa}{M}\right),
\end{equation}
the transport contribution ($\mathcal{T}(g_e)$) being dominant at this stage.

Since $\mathcal{T}$ evaluates the derivatives of its argument rather than the
argument itself, the bound \eqref{eq:first_pass} on $g_e$ alone would leave
$\mathcal{T}(g_e)$ unchanged, and a further substitution into
\eqref{eq:ge_implicit} would merely reproduce the same estimate. The
derivatives of $g_e$ must therefore be estimated as well, from their own
equations. Differentiating \eqref{eq:ge_balance} with respect to $x$ and $v$
gives
\begin{align}
    \Big(\nu_{ee}+\kappa\partial_t-\tfrac{\kappa}{M}\mathcal{T}\Big)
    \partial_x g_e
    &= \frac{\kappa}{M}\,\partial_x S_e
       - (\partial_x \nu_{ee})\, g_e
       + \frac{\kappa}{M}\,[\partial_x,\mathcal{T}]\, g_e ,
       \label{eq:dx_equation} \\[2pt]
    \Big(\nu_{ee}+\kappa\partial_t-\tfrac{\kappa}{M}\mathcal{T}\Big)
    \partial_v g_e
    &= \frac{\kappa}{M}\,\partial_v S_e
       + \frac{\kappa}{M}\,[\partial_v,\mathcal{T}]\, g_e ,
       \label{eq:dv_equation}
\end{align}
where $[\partial_\bullet,\mathcal{T}] := \partial_\bullet\circ\mathcal{T} - \mathcal{T}\circ\partial_\bullet$ is operator commutator, with
\begin{align} \label{eq:commutator_x}
    &\begin{aligned}
        [\partial_x,\mathcal{T}]\, g
    &= \left(\partial_x \Liu{\Pi}_{\mathcal{M}}\right) (v \partial_x g) + \eta\, (\partial_x E)\, \partial_v g \\
    &=M \, \frac{\,\partial_x u_e}{T_{e,0}} \left( \frac{(v-Mu_e)^2}{T_{e,0}} - 1 \right)
      \frac{\mathcal{M}_e}{n_e} \big\langle (v-Mu_e)\, v\, \partial_x g \big\rangle\\
      &\quad+ \eta\, (\partial_x E)\, \partial_v g ,
    \end{aligned}\\
    &[\partial_v,\mathcal{T}]\, g =  -(\mathbb{I}-\Pi_{\mathcal{M}})\big(\partial_x g\big) +\Liu{(\partial_v\Pi_{\mathcal{M}})
    \bigl(v\partial_x g\bigr)}.
\end{align}
\Liu{Differentiating the definition of $\Pi_{\mathcal{M}}$ and using the moment
constraints $\langle g_e\rangle=\langle vg_e\rangle=0$, one obtains
\begin{equation*}
    (\partial_v\Pi_{\mathcal{M}})
    \bigl(v\partial_xg_e\bigr)
    =
    \frac{\partial_x\langle v^2g_e\rangle}
         {n_eT_{e,0}}
    \left(
        1-\frac{(v-Mu_e)^2}{T_{e,0}}
    \right)\mathcal{M}_e .
\end{equation*}
}

\Del{This term involves no derivative of $g$ beyond first order and is bounded on $L^\infty_v(\Omega_v)$ under the
uniform regularity assumptions. It is therefore $\mathcal{O}(1)$ in the first pass and, once
$\partial_x g_e = \mathcal{O}(\kappa/M)$ has been established, becomes
$\mathcal{O}(\kappa/M)$; after multiplication by $\kappa/M$ in
\eqref{eq:dv_equation} it contributes only $\mathcal{O}(\kappa^2/M^2) =
o(\kappa)$ and does not affect the final estimate.}

 Both equations \Del{\eqref{eq:dx_equation}--\eqref{eq:dv_equation} are governed by the same operator as equation} \eqref{eq:ge_balance}, and their right-hand sides involve only the differentiated source $S_e$, which is $\mathcal{O}(M)$ by \eqref{eq:source_exact}, together with lower-order derivatives of $g_e$. Estimating them in turn, the hierarchy being processed upwards from $g_e$, the first pass gives
\begin{equation} \label{eq:first_pass_der}
    \partial_x g_e = \mathcal{O}\!\left(\frac{\kappa}{M}\right),
    \quad
    \partial_v g_e = \mathcal{O}\!\left(\frac{\kappa}{M}\right).
\end{equation}
These estimates are sufficient to obtain the optimal bound on $g_e$ thanks to a second pass.

\item \emph{Second pass.} Substituting \eqref{eq:first_pass_der} back into
\eqref{eq:ge_implicit}, the transport contribution now reads
\begin{equation*}
    \frac{\kappa}{M}\,\mathcal{T}(g_e) = \mathcal{O}\!\left(\frac{\kappa^2}{M^2}\right) = o(\kappa),
\end{equation*}
by the coupled scaling $\kappa/M^2 \to 0$: it has become negligible with
respect to the source. The remaining unsteady contribution
$\kappa\,\partial_t g_e$ is $\mathcal{O}(\kappa)$, of the same order as the
source, so that
\begin{equation} \label{eq:ge_order_kappa}
    g_e = \mathcal{O}(\kappa) .
\end{equation}
The optimal bound for $g_e$ is therefore obtained after two passes, since it cannot be improved due to the source term $S_e$ and the unsteady operator $\kappa\,\partial_t g_e$  contributions.
\end{enumerate}

Repeating the same substitution in \eqref{eq:dx_equation}, where the commutator
term is likewise reduced to
$\frac{\kappa}{M}\,\mathcal{O}(\kappa/M) = o(\kappa)$ and the term
$-(\partial_x\nu_{ee})g_e$ is $\mathcal{O}(\kappa)$ by
\eqref{eq:ge_order_kappa}, we obtain
\begin{equation} \label{eq:dxge_order_kappa}
    \partial_x g_e = \mathcal{O}(\kappa)
    \quad \text{uniformly in } L^\infty_v(\Omega_v).
\end{equation}
Assumption~(iv) guarantees that this
quasi-stationary balance also holds at $t = 0$: for data that are not well
prepared, the homogeneous dynamics
$\partial_t g_e = -\nu_{ee}\,g_e/\kappa$ would first relax $g_e$ towards
$\mathcal{O}(\kappa)$ over the short time $t \sim \kappa/\nu_{ee}$, and
\eqref{eq:ge_order_kappa} would only hold outside that initial layer.
\item \Del{The optimal bound on the spatial derivative of the kinetic pressure can
now be deduced. Unlike the crude estimates used in the previous step, where
$\partial_x\langle v^2 g_e\rangle$ only had to be controlled well enough to
close the commutator contributions, we now seek the sharp order of this
quantity, which is the conclusion of the lemma.}
Since differentiation in $x$ commutes with integration in $v$, and
$\Omega_v = [-V_{\max}, V_{\max}]$ is bounded, owing to equation~\eqref{eq:dxge_order_kappa}
\begin{equation*}
    \left| \partial_x \langle v^2 g_e \rangle \right|
    = \left| \int_{\Omega_v} v^2\, \partial_x g_e \, dv \right|
    \le 2 V_{\max}^3\,
        \|\partial_x g_e\|_{L^\infty_v(\Omega_v)}
    = \mathcal{O}(\kappa)
\end{equation*}
\end{enumerate}

Under the coupled scaling $\kappa = o(M^2)$ we
finally obtain $\partial_x \langle v^2 g_e \rangle = o(M^2)$, which completes
the proof.
\end{proof}

\begin{proof}[Proof of Proposition~\ref{prop2.4}] \label{prop2.4pf}
We establish sequentially the exact equivalence for $M > 0$, the formal limit $M \to 0$, and the structural non-degeneracy.
\begin{enumerate}
    \item \emph{Equivalence for $M > 0$ under initial compatibility \eqref{eq:aug:compat}:}
Taking the spatial derivative $\partial_x$ of the momentum equation \eqref{eq:aug:2} yields:
\begin{equation} \label{eq:pf:dt_flux}
    \partial_x \big( \partial_t (n_e u_e) \big) = -\partial^2_{xx} \big( n_e u_e^2 + L \big).
\end{equation}
Substituting relation \eqref{eq:pf:dt_flux} into the density wave equation \eqref{eq:aug:wave} leads to:
\begin{equation*}
    \partial^2_{tt} n_e + \partial_x \big( \partial_t (n_e u_e) \big) = 0 \iff \partial_t \Big( \partial_t n_e + \partial_x (n_e u_e) \Big) = 0,
\end{equation*}
%
%Integrating this relation with respect to time $t$ gives:
%\begin{equation*}
%    \partial_t n_e(t,x) + \partial_x \big( n_e(t,x) u_e(t,x) \big) = C(x),
%\end{equation*}
%where $C(x)$ is an arbitrary time-independent integration function. Evaluating this identity at $t = 0$ and applying 
%\begin{equation*}
%    C(x) = \left. \partial_t n_e \right|_{t=0} + \partial_x \big( n_e^0 u_e^0 \big) = 0.
%\end{equation*}
%Thus, $C(x) \equiv 0$, and the exact mass conservation equation $\partial_t n_e + \partial_x (n_e u_e) = 0$ is recovered for all $t \ge 0$. Next, substituting $M^2 \partial_x L$ from \eqref{eq:aug:3} into \eqref{eq:aug:2} and multiplying by $M^2$ recovers the exact original momentum balance equation:
which, thanks to the initial compatibility condition \eqref{eq:aug:compat} and substituting $M^2 \partial_x L$ from \eqref{eq:aug:3} into \eqref{eq:aug:2} and multiplying by $M^2$ recovers the exact original momentum balance equation:
\begin{equation*}
    M^2 \Big( \partial_t (n_e u_e) + \partial_x (n_e u_e^2) \Big) + T_{e,0} \partial_x n_e - \eta n_e \partial_x \phi + \partial_x \langle v^2 g_e \rangle = 0.
\end{equation*}
This establishes strict equivalence with the original micro--macro model for all $M > 0$.
\item \emph{Limit $M \to 0$ and equivalence with Proposition~\ref{prop2.3}:}
We inject the Hilbert expansions \eqref{eq:hilbert_expansions} and $L = \sum_{k=0}^\infty (M^2)^k L_k$ into system \eqref{eq:aug:wave}--\eqref{eq:aug:5}. Lemma~\ref{lem:ge_estimate_1d} ensures $\partial_x \langle v^2 g_e \rangle = o(M^2)$.
At order $\mathcal{O}(1)$ in \eqref{eq:aug:3} we obtain $T_{e,0} \partial_x n_{e,0} - \eta n_{e,0} \partial_x \phi_0 = 0$, which is the Boltzmann equilibrium relation \eqref{eq:M0:2}.

Identifying terms of order $M^2$  in equation~\eqref{eq:aug:3} gives $\partial_x L_0 = T_{e,0} \partial_x n_{e,1} - \eta \big( n_{e,0} \partial_x \phi_1 + n_{e,1} \partial_x \phi_0 \big)$. 
Now considering order $\mathcal{O}(1)$ in \eqref{eq:aug:wave}--\eqref{eq:aug:2}: Taking $M \to 0$ yields the limit wave system:
    \begin{equation*}
        \partial_{tt} n_{e,0} - \partial_{xx} \big( n_{e,0} u_{e,0}^2 + L_0 \big) = 0, \quad \partial_t (n_{e,0} u_{e,0}) + \partial_x (n_{e,0} u_{e,0}^2) + \partial_x L_0 = 0.
    \end{equation*}
    This pair combined with the leading-order compatibility condition $\left. \partial_t n_{e,0} \right|_{t=0} = -\partial_x (n_{e,0}^0 u_{e,0}^0)$ is strictly equivalent to the coupled continuity and momentum limit equations \eqref{eq:M0:1} and \eqref{eq:momentum:Hilbert} of Proposition~\ref{prop2.3}.
\item \Del{\emph{Non-degeneracy $M \ge 0$ and $\lambda \ge 0$:}} Formally taking the limit $M \to 0$ in the augmented micro--macro system yields the following closed set of equations:
\begin{align}
    &-\partial^2_{xx} L = -\partial^2_{tt} n_e +\partial^2_{xx} \big( n_e u_e^2 \big), \label{eq:aug:wave:0}\\
    &\partial_t (n_e u_e) + \partial_x (n_e u_e^2) + \partial_x L = 0, \label{eq:aug:2:0}\\
    &- \eta \partial_x \left( n_e \partial_x \phi\right) = T_{e,0} \partial^2_{xx} n_e, \label{eq:aug:3:0}\\
    &g_e = 0, \label{eq:aug:4:0}\\
    &n_e = n_i + \lambda^2 \eta \partial^2_{xx} \phi. \label{eq:aug:5:0}
\end{align}
In this system, equation \eqref{eq:aug:3:0} is obtained by differentiating the spatial Boltzmann relation with respect to $x$ in order to formulate a second-order elliptic equation for the potential or density. This modification is proposed on purpose to facilitate the numerical treatment of this augmented system. 

Unlike the original formulation, the momentum equation \eqref{eq:aug:2:0} remains fully evolutionary, allowing the momentum $n_e u_e$ to be computed dynamically. The kinetic deviation $g_e$ vanishes identically according to \eqref{eq:aug:4:0}, while the auxiliary field $L$ is uniquely determined by solving the spatial elliptic equation \eqref{eq:aug:wave:0}.

Regarding the electron density $n_e$ and the electric potential $\phi$, two distinct frameworks are naturally covered depending on the scaled Debye length $\lambda$:\\
\emph{Non-quasi-neutral regime ($\lambda > 0$):} The electric potential $\phi$ is determined by solving Poisson equation \eqref{eq:aug:5:0}, while the density $n_e$ is evolved dynamically through the wave formulation \eqref{eq:aug:wave:0}--\eqref{eq:aug:2:0}.\\
\emph{Quasi-neutral regime ($\lambda = 0$):} Poisson equation \eqref{eq:aug:5:0} degenerates into the algebraic quasi-neutral constraint $n_e = n_i$, which directly yields the density, whereas the potential $\phi$ is retrieved by solving the elliptic equation \eqref{eq:aug:3:0}.

Note that this non-singular behavior in the quasi-neutral limit holds uniformly for $M > 0$ as well. Consequently, the augmented micro--macro system is uniformly well-posed in both the low-Mach ($M \to 0$) and quasi-neutral ($\lambda \to 0$) limits.
\end{enumerate}
\end{proof}

\bibliographystyle{abbrv}
\bibliography{references}
\end{document}